\documentclass[12pt,a4paper,reqno]{amsart}
\usepackage{amsmath}
\usepackage{amsthm}
\usepackage{amsfonts}
\usepackage{amssymb}
\usepackage{color}
\usepackage{mathrsfs}
\usepackage{graphicx}

\numberwithin{equation}{section}

\theoremstyle{plain}
\newtheorem{theorem}{Theorem}

\theoremstyle{plain}
\newtheorem{lemma}{Lemma}[section]

\theoremstyle{plain}
\newtheorem*{op1}{Open Problem}

\theoremstyle{remark}
\newtheorem*{remark}{Remark}

\theoremstyle{remark}
\newtheorem*{remarka}{Remark on the value of $\alpha$}

\theoremstyle{remark}
\newtheorem*{remarkb}{Remark on the value of $\beta$}

\theoremstyle{definition}
\newtheorem{example}{Example}[section]

\theoremstyle{definition}
\newtheorem*{part1}{Part 1}

\theoremstyle{definition}
\newtheorem*{part2}{Part 2}

\theoremstyle{plain}
\newtheorem*{pa}{Property A}

\theoremstyle{plain}
\newtheorem*{pb}{Property B}

\def\bfp{\mathbf{p}}

\def\bfs{\mathbf{s}}
\def\bft{\mathbf{t}}

\def\bfv{\mathbf{v}}
\def\bfw{\mathbf{w}}
\def\bfx{\mathbf{x}}

\def\dd{\mathrm{d}}
\def\ee{\mathrm{e}}
\def\ii{\mathrm{i}}

\def\eps{\varepsilon}

\def\Qq{\mathbb{Q}}
\def\Rr{\mathbb{R}}

\def\Zz{\mathbb{Z}}

\def\AAA{\mathcal{A}}
\def\BBB{\mathcal{B}}
\def\CCC{\mathcal{C}}

\def\HHH{\mathcal{H}}
\def\III{\mathcal{I}}

\def\LLL{\mathcal{L}}
\def\MMM{\mathcal{M}}

\def\PPP{\mathcal{P}}

\def\SSS{\mathcal{S}}

\def\WWW{\mathcal{W}}
\def\XXX{\mathcal{X}}
\def\YYY{\mathcal{Y}}

\def\frakd{\mathfrak{d}}

\def\frakI{\mathfrak{I}}
\def\frakJ{\mathfrak{J}}

\def\frakM{\mathfrak{M}}

\def\frakR{\mathfrak{R}}

\def\BBBB{\mathscr{B}}

\def\NNNN{\mathscr{N}}

\DeclareMathOperator{\Circle}{circle}
\DeclareMathOperator{\Disk}{disk}
\DeclareMathOperator{\bad}{bad}
\DeclareMathOperator{\violator}{violator}
\DeclareMathOperator{\cs}{cs}

\renewcommand{\le}{\leqslant}
\renewcommand{\ge}{\geqslant}

\title[Uniformity of geodesic flow]
{Uniformity of geodesic flow\\
in non-integrable 3-manifolds}

\author[Beck]{J. Beck}

\address{Department of Mathematics, Hill Center for the Mathematical Sciences, Rutgers University, Piscataway NJ 08854, USA}

\email{jbeck@math.rutgers.edu}

\author[Chen]{W.W.L. Chen}

\address{School of Mathematical and Physical Sciences, Faculty of Science and Engineering, Macquarie University, Sydney NSW 2109, Australia}

\email{william.chen@mq.edu.au}

\author[Yang]{Y. Yang}

\address{School of Science, Beijing University of Posts and Telecommunications, Beijing 100876, China}

\email{yangyx@bupt.edu.cn}

\begin{document}

\maketitle

\thispagestyle{empty}

%
%

\section{Introduction and the splitting method}\label{sec1}

Very little is known about non-integrable dynamical systems in dimensions greater than~$2$.
A perfect illustration of this sad fact is the following open problem which is the $3$-dimensional non-integrable analogue
of the Kronecker--Weyl equidistribution theorem concerning geodesic flow on the unit torus $[0,1)^3$.

A vector $\bfv=(v_1,v_2,v_3)\in\Rr^3$ is called a \textit{Kronecker direction} in $3$-space if the coordinates $v_1,v_2,v_3$
are linearly independent over~$\Qq$.

\begin{op1}
Suppose that $\MMM$ is a finite polycube translation $3$-manifold.
Is it true that any half-infinite geodesic in $\MMM$ with a Kronecker direction and which does not hit a singular point of $\MMM$ is uniformly distributed?
\end{op1}

The answer to the corresponding $2$-dimensional problem is given by the Gutkin--Veech theorem,
but the method does not seem to work in higher dimensions.
In $3$ dimensions, even the simpler problem of density is wide open.

Here we introduce a new geometric method which enables us to solve the Open Problem
for infinitely many finite polycube translation $3$-manifolds which, in the first instance,
satisfy a very special condition which makes them
\textit{essentially integrable} in one of the $3$ directions.
We then extend our argument in Sections~\ref{sec8}--\ref{sec9}
to include more finite polycube translation $3$-manifolds.

We call our new approach the \textit{splitting method}.
It uses the Birkhoff ergodic theorem, and leads to time-qualitative uniformity.
It is an open problem whether time-quantitative uniformity can be established in these situations.

To illustrate our method and ideas, we consider in the first instance the L-solid translation $3$-manifold,
the simplest non-integrable example in $3$-dimensions.
This $3$-manifold consists of $3$ atomic cubes arranged in the shape of the letter~L, with a \textit{top cube}, a \textit{middle cube}
and a \textit{right cube}, together with face identification given by perpendicular translation.
The picture on the right in Figure~1.1 shows the $Y$-faces which are perpendicular to the $y$-axis.
The two faces $Y_2$ are identified with each other, and the two faces $Y_3$ are identified with each other.

\begin{displaymath}
\begin{array}{c}
\includegraphics[scale=0.8]{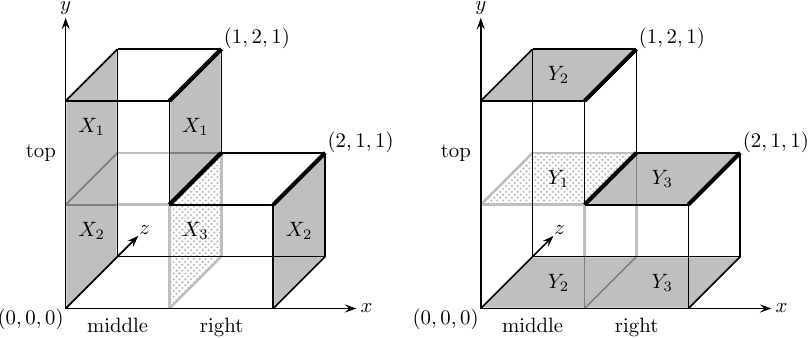}
\vspace{3pt}\\
\mbox{Figure 1.1: the $X$-faces and $Y$-faces of the L-solid translation $3$-manifold}
\end{array}
\end{displaymath}

\begin{theorem}\label{thm1}
Any half-infinite geodesic in the L-solid translation $3$-manifold with a Kronecker direction $\bfv\in\Rr^3$
is uniformly distributed unless it hits a singularity.
\end{theorem}

Suppose that $\bfv=(v_1,v_2,v_3)\in\Rr^3$ is a Kronecker direction in $3$-space.
Linear independence guarantees that the coordinates $v_1,v_2,v_3$ are non-zero.
It is convenient to denote the Kronecker direction instead by the parallel vector $\bfv=(\alpha,1,\beta)$,
where $\alpha=v_1/v_2$, $\beta=v_3/v_2$ and $\alpha/\beta$ are all irrational.
For simplicity, we assume that $\alpha>0$ and $\beta>0$.
We can further assume that $\alpha<1$.
If $\alpha>1$, then we simply interchange the roles of $x$ and~$y$.

We consider the \textit{discretization} of the $\bfv$-flow in the L-solid translation $3$-manifold relative to the $Y$-faces,
where the flow hits the faces $Y_1,Y_2,Y_3$.
Let
\begin{equation}\label{eq1.1}
T_{\alpha,\beta}:\YYY=Y_1\cup Y_2\cup Y_3\to\YYY
\end{equation}
denote the relevant discrete transformation defined by consecutive hitting points.

Under our assumptions on the parameters $\alpha$ and~$\beta$, the $\bfv$-flow encounters split singularities in the form
of edges of the L-solid translation $3$-manifold, highlighted in Figure~1.1 in bold.
For instance, consider a geodesic segment in the $\bfv$-direction that hits the common edge of the faces $Y_1$ and~$Y_3$.
If we move this geodesic segment marginally to the left, then it crosses the face~$Y_1$, quickly hits the right face~$X_1$
and then jumps to the left face~$X_1$.
If we move this geodesic segment marginally to the right, then it crosses the face~$X_3$, quickly hits the top face~$Y_3$
and then jumps to the bottom face~$Y_3$.

Consider now the image under $T_{\alpha,\beta}^{-1}$ of the singular edge of the L-solid translation $3$-manifold which is the
common edge of the faces $Y_1$ and ~$Y_3$.
It is a line segment in the $z$-direction on the face~$Y_2$, as shown in Figure~1.2.

\begin{displaymath}
\begin{array}{c}
\includegraphics[scale=0.8]{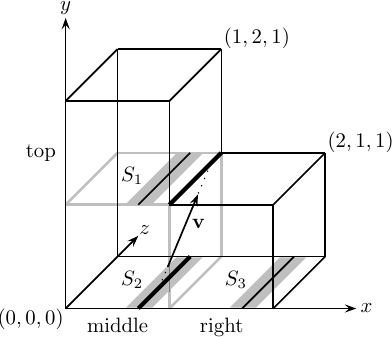}
\vspace{3pt}\\
\mbox{Figure 1.2: a strip round the inverse image of a singular edge}
\end{array}
\end{displaymath}

Consider a narrow strip neighbourhood $S_2$ of this line segment, shown in Figure~1.2 as the gray strip on the face~$Y_2$.
The image of this strip under $T_{\alpha,\beta}$ splits along the singular edge.
Our splitting method concerns making use of such splits.

\begin{remark}
Note that the front $Z$-faces on the plane $z=0$ and the corresponding back $Z$-faces on the plane $z=1$ are pairwise identified,
and none of the edges of these faces is a singular edge, so the L-solid translation $3$-manifold is essentially integrable in the
$z$-direction.
\end{remark}

The discretization $T_{\alpha,\beta}$ on $\YYY$ defined by \eqref{eq1.1} is area preserving.
Our first goal is to prove that it is also \textit{ergodic} on~$\YYY$.
Suppose, on the contrary, that it is not ergodic.
Then there exist two disjoint $T_{\alpha,\beta}$-invariant subsets $\WWW$ (white) and $\SSS$ (silver) of $\YYY$
such that $\YYY=\WWW\cup\SSS$ and
\textcolor{white}{xxxxxxxxxxxxxxxxxxxxxxxxxxxxxx}
\begin{equation}\label{eq1.2}
0<\lambda_2(\WWW)\le\lambda_2(\SSS)<3,
\end{equation}
where $\lambda_2$ denotes the $2$-dimensional Lebesgue measure.
We show that this leads to a contradiction.

Consider the projection of $\YYY$ to the unit torus $[0,1)^2$, given by
\begin{equation}\label{eq1.3}
\YYY\to[0,1)^2:(x,y,z)\mapsto(\{x\},\{z\}),
\end{equation}
where we take the fractional parts of the first and the third coordinates and omit reference to the second coordinate.
Then for every point $P\in[0,1)^2$, there are precisely $3$ distinct points $P_1,P_2,P_3\in\YYY$ which have projection image~$P$.
Let
\begin{equation}\label{eq1.4}
f_\WWW(P)=\vert\{P_1,P_2,P_3\}\cap\WWW\vert
\quad\mbox{and}\quad
f_\SSS(P)=\vert\{P_1,P_2,P_3\}\cap\SSS\vert
\end{equation}
denote the number of these $3$ points that fall into $\WWW$ and $\SSS$ respectively.
We refer to $f_\WWW$ and $f_\SSS$ as the multiplicity functions of the invariant sets $\WWW$ and $\SSS$ respectively,
and they are non-negative integer valued functions defined on the unit torus $[0,1)^2$ that satisfy the condition
\begin{equation}\label{eq1.5}
f_\WWW(P)+f_\SSS(P)=3
\quad
\mbox{for almost all $P\in[0,1)^2$}.
\end{equation}
We also consider the corresponding projection of
\begin{displaymath}
T_{\alpha,\beta}:\YYY\to\YYY
\quad\mbox{to}\quad
T_0:[0,1)^2\to[0,1)^2,
\end{displaymath}
which is simply the $(\alpha,\beta)$-shift on the unit torus $[0,1)^2$.

\begin{remark}
Strictly speaking, $\WWW$ and $\SSS$ are subsets of~$\YYY$.
However, it is convenient to view the pair as a $2$-colouring of the set $\YYY$ with colours $\WWW$ and~$\SSS$,
and not make any distinction between subset and colour.
\end{remark}

Since $\bfv$ is a Kronecker direction, it is not difficult to show that $T_0$ is ergodic.
Indeed, one way is to derive it from the Kronecker--Weyl equidistribution theorem on the unit torus.
Here we give a simpler proof by using the more elementary Kronecker density theorem which is valid in every dimension $d\ge1$.
However, we restrict our discussion to the case $d=2$.

Let $\bfw=(\alpha,\beta)\in\Rr^2$, where $\alpha,\beta$ are linearly independent over~$\Qq$,
and let $T_0$ be the $\bfw$-shift on the unit torus $[0,1)^2$.
If $T_0$ is not ergodic, then there is a $\bfw$-shift-invariant measurable subset $A\subset[0,1)^2$ such that
\begin{equation}\label{eq1.6}
0<\lambda_2(A)<1
\quad\mbox{and}\quad
A+\bfw=A,
\end{equation}
where $A+\bfw$ denotes the image of $A$ under~$T_0$.

\begin{lemma}\label{lem11}
Suppose that the set $A\subset[0,1)^2$ is measurable.
Then
\begin{equation}\label{eq1.7}
\int_{[0,1)^2}\lambda_2(A\cap(A+\bft))\,\dd\bft=(\lambda_2(A))^2.
\end{equation}
\end{lemma}

\begin{proof}
For any set $B\subset[0,1)^2$, let $\chi_B$ denote the characteristic function of~$B$.
Then
\begin{align}
&
\int_{[0,1)^2}\lambda_2(A\cap(A+\bft))\,\dd\bft
=\int_{[0,1)^2}\left(\int_A\chi_{A+\bft}(\bfx)\,\dd\bfx\right)\dd\bft
\nonumber
\\
&\quad
=\int_{[0,1)^2}\left(\int_A\chi_{A-\bfx}(-\bft)\,\dd\bfx\right)\dd\bft
=\int_A\left(\int_{[0,1)^2}\chi_{A-\bfx}(-\bft)\,\dd\bft\right)\dd\bfx
\nonumber
\\
&\quad
=\int_A\lambda_2(A-\bfx)\,\dd\bfx
=\int_A\lambda_2(A)\,\dd\bfx
=(\lambda_2(A))^2.
\nonumber
\end{align}
The change in the order of integration is justified by Fubini's theorem, and this completes the proof.
\end{proof}

The identity \eqref{eq1.7} represents an average.
Let $\eps>0$ be chosen to be sufficiently small.
Then there exists $\bft_0\in[0,1)^2$ such that
\begin{displaymath}
\lambda_2(A\cap(A+\bft_0))<(\lambda_2(A))^2+\eps,
\end{displaymath}
and the Kronecker density theorem implies that there exists a positive integer $n_0$ such that the multiple $n_0\bfw$
gets so close to $\bft_0$ modulo $[0,1)^2$ that
\begin{displaymath}
\lambda_2(A\cap(A+n_0\bfw))<\lambda_2(A\cap(A+\bft_0))+\eps.
\end{displaymath}
Then it follows that
\textcolor{white}{xxxxxxxxxxxxxxxxxxxxxxxxxxxxxx}
\begin{equation}\label{eq1.8}
\lambda_2(A\cap(A+n_0\bfw))<(\lambda_2(A))^2+2\eps.
\end{equation}
On the other hand, the identity on the right in \eqref{eq1.6} clearly gives $A+n_0\bfw=A$, so \eqref{eq1.8} becomes
$\lambda_2(A)<(\lambda_2(A))^2+2\eps$.
However, this inequality cannot possibly hold when $\eps>0$ is sufficiently small, since $0<\lambda_2(A)<1$.
The contradiction establishes the ergodicity of~$T_0$.

The Birkhoff ergodic theorem now guarantees that both $f_\WWW$ and $f_\SSS$ are constant integer valued functions.
It then follows from \eqref{eq1.2} and \eqref{eq1.5} that
\begin{equation}\label{eq1.9}
f_\WWW(P)=1
\quad\mbox{and}\quad
f_\SSS(P)=2
\quad
\mbox{for almost all $P\in[0,1)^2$}.
\end{equation}

Let us return to the narrow strip $S_2$ on the face~$Y_2$, as shown in Figure~1.2.
We denote by $S_1$ and $S_3$ narrow strips on the lower faces $Y_1$ and $Y_3$ respectively, where the images on $[0,1)^2$
of $S_1,S_2,S_3$ under the projection \eqref{eq1.3} coincide.
In other words, the strips $S_1,S_2,S_3$ are in the same relative positions on the faces $Y_1,Y_2,Y_3$ respectively.

Using line segments in the $x$-direction, we can divide the strips $S_1,S_2,S_3$ into small congruent rectangles in a natural way.
We refer to them as \textit{small special rectangles} of the decomposition of the strips $S_1,S_2,S_3$.
Note that we have so far not specified the dimensions of these rectangles.

Let $R_1,R_2,R_3$ denote small special rectangles of the decomposition of the strips $S_1,S_2,S_3$
on the faces $Y_1,Y_2,Y_3$ respectively, with the condition that their images on $[0,1)^2$ under the projection
\eqref{eq1.3} coincide, so that they are in the same relative positions on the faces $Y_1,Y_2,Y_3$ respectively.
It is clear that the image of each of $R_1,R_2,R_3$ under $T_{\alpha,\beta}$ splits along the singular edges $E_1,E_2,E_3$
into congruent halves, as shown in the side views given in Figure~1.3 where the $z$-direction is suppressed.
Here, for each $\sigma=1,2,3$, the rectangle $R_\sigma$ splits into the left half $R_\sigma^-$ and the right half~$R_\sigma^+$,
and it is clear that their images under $T_{\alpha,\beta}$ may lie on distinct $Y$-faces.

\begin{displaymath}
\begin{array}{c}
\includegraphics[scale=0.8]{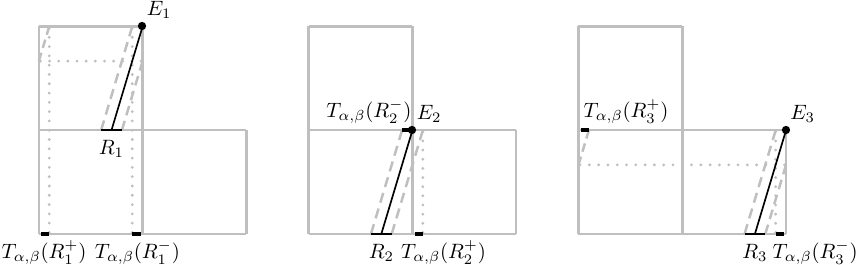}
\vspace{3pt}\\
\mbox{Figure 1.3: side view of L-solid translation $3$-manifold showing}
\\
\mbox{the rectangles $R_1,R_2,R_3$ and their images under $T_{\alpha,\beta}$}
\end{array}
\end{displaymath}

Figure~1.4 is a summary of the pictures in Figure~1.3.
The union of the pairs
\begin{displaymath}
T_{\alpha,\beta}(R_3^-)\cup T_{\alpha,\beta}(R_1^+),
\quad
T_{\alpha,\beta}(R_2^-)\cup T_{\alpha,\beta}(R_3^+),
\quad
T_{\alpha,\beta}(R_1^-)\cup T_{\alpha,\beta}(R_2^+)
\end{displaymath}
are rectangles the same size as the rectangles $R_1,R_2,R_3$, each split down the middle into congruent halves by a line in the $z$-direction.

\begin{displaymath}
\begin{array}{c}
\includegraphics[scale=0.8]{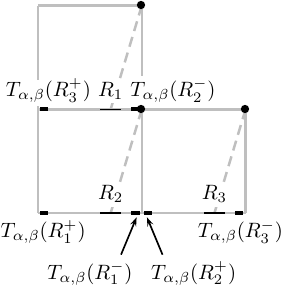}
\vspace{3pt}\\
\mbox{Figure 1.4: side view of L-solid translation $3$-manifold showing}
\\
\mbox{the rectangles $R_1,R_2,R_3$ and their images under $T_{\alpha,\beta}$}
\end{array}
\end{displaymath}

The situation is somewhat clearer if we study the images of $R_1,R_2,R_3$ under $T_{\alpha,\beta}^2$,
where the second application of $T_{\alpha,\beta}$ does not produce any further splitting,
as shown in Figure~1.5.

\begin{displaymath}
\begin{array}{c}
\includegraphics[scale=0.8]{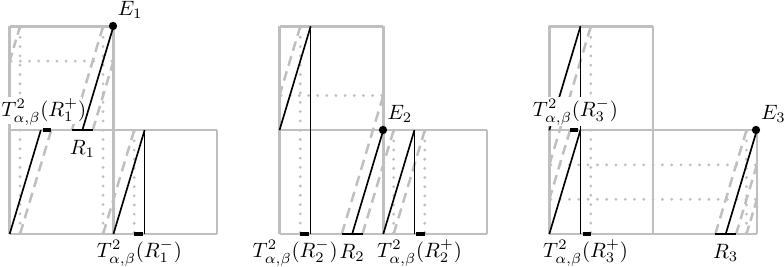}
\vspace{3pt}\\
\mbox{Figure 1.5: side view of L-solid translation $3$-manifold showing}
\\
\mbox{the rectangles $R_1,R_2,R_3$ and their images under $T_{\alpha,\beta}^2$}
\end{array}
\end{displaymath}

Figure~1.6 is a summary of the pictures in Figure~1.5.
The union of the pairs
\begin{equation}\label{eq1.10}
\begin{array}{l}
R^*_1=T_{\alpha,\beta}^2(R_3^-)\cup T_{\alpha,\beta}^2(R_1^+),
\vspace{3pt}\\
R^*_2=T_{\alpha,\beta}^2(R_2^-)\cup T_{\alpha,\beta}^2(R_3^+),
\vspace{3pt}\\
R^*_3=T_{\alpha,\beta}^2(R_1^-)\cup T_{\alpha,\beta}^2(R_2^+),
\end{array}
\end{equation}
are rectangles on the faces $Y_1,Y_2,Y_3$ respectively, the same size as $R_1,R_2,R_3$,
each split into congruent halves by a line in the $z$-direction.
Note that the images on $[0,1)^2$ of $R^*_1,R^*_2,R^*_3$ under the projection \eqref{eq1.3} coincide,
so that $R^*_1,R^*_2,R^*_3$ are in the same relative positions on the faces $Y_1,Y_2,Y_3$ respectively.

\begin{displaymath}
\begin{array}{c}
\includegraphics[scale=0.8]{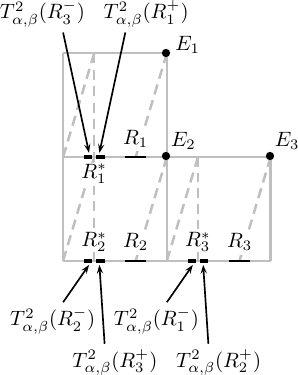}
\vspace{3pt}\\
\mbox{Figure 1.6: side view of L-solid translation $3$-manifold showing}
\\
\mbox{the rectangles $R_1,R_2,R_3$ and their images under $T_{\alpha,\beta}^2$}
\end{array}
\end{displaymath}

We need the following measure theoretic lemma.

\begin{lemma}\label{lem12}
There exists a small special rectangle $R_1$ on the face $Y_1$ of the L-solid translation $3$-manifold such that
the following two conditions are satisfied:

\emph{(i)}
The set $R_1$ satisfies
\begin{equation}\label{eq1.11}
\frac{\lambda_2(R_1\cap\WWW)}{\lambda_2(R_1)}>\frac{99}{100}
\quad\mbox{or}\quad
\frac{\lambda_2(R_1\cap\SSS)}{\lambda_2(R_1)}>\frac{99}{100}.
\end{equation}

\emph{(ii)}
There are analogues of \emph{(i)} for the small special rectangles $R_2$ and $R_3$ on the faces $Y_2$ and $Y_3$
respectively of the L-solid translation $3$-manifold, defined such that the images on $[0,1)^2$ of $R_1,R_2,R_3$
under the projection \eqref{eq1.3} coincide, as well as for the rectangles $R^*_1,R^*_2,R^*_3$ defined by \eqref{eq1.10}.
\end{lemma}

These rectangles are illustrated in Figure~1.7.

\begin{displaymath}
\begin{array}{c}
\includegraphics[scale=0.8]{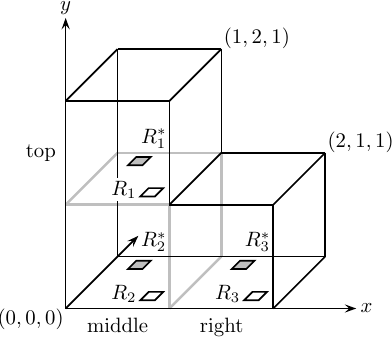}
\vspace{3pt}\\
\mbox{Figure 1.7: the rectangles $R_1,R_2,R_3$ and $R^*_1,R^*_2,R^*_3$}
\end{array}
\end{displaymath}

The choice of $99/100$ in \eqref{eq1.11} is accidental.
Note also that Lemma~\ref{lem12} somewhat resembles the Lebesgue density theorem.
We include the proof in Section~\ref{sec3}.

Recall that we have assumed that $T_{\alpha,\beta}$ is not ergodic, and this implies that there exist disjoint $T_{\alpha,\beta}$-invariant
subsets $\WWW,\SSS\subset\YYY$ such that $\YYY=\WWW\cup\SSS$ and \eqref{eq1.4} and \eqref{eq1.9} are satisfied.
Using Lemma~\ref{lem12}, we are now ready to deduce a contradiction.

Lemma~\ref{lem12} ensures that each rectangle $R_\sigma$, $\sigma=1,2,3$, has a dominant colour~$\CCC_\sigma$,
and each rectangle $R^*_\sigma$, $\sigma=1,2,3$, has a dominant colour~$\CCC^*_\sigma$,
where each $\CCC_\sigma$ or $\CCC^*_\sigma$ is white or silver.
We now use Figures 1.3, 1.5 and~1.6, and also note that since $\WWW$ and $\SSS$ are $T_{\alpha,\beta}$-invariant,
they are also $T_{\alpha,\beta}^2$-invariant.

The image of the rectangle $R_1$ under $T_{\alpha,\beta}$ splits along the edge $E_1$ into congruent halves.
The image of~$R_1^-$, the left half of~$R_1$, under $T_{\alpha,\beta}^2$ forms the left half of~$R^*_3$.
The image of~$R_1^+$, the right half of~$R_1$, under $T_{\alpha,\beta}^2$ forms the right half of~$R^*_1$.
Thus
\begin{equation}\label{eq1.12}
\CCC_1=\CCC^*_1=\CCC^*_3.
\end{equation}
The image of the rectangle $R_2$ under $T_{\alpha,\beta}$ splits along the edge $E_2$ into congruent halves.
The image of~$R_2^-$, the left half of~$R_2$, under $T_{\alpha,\beta}^2$ forms the left half of~$R^*_2$.
The image of~$R_2^+$, the right half of~$R_2$, under $T_{\alpha,\beta}^2$ forms the right half of~$R^*_3$.
Thus
\begin{equation}\label{eq1.13}
\CCC_2=\CCC^*_2=\CCC^*_3.
\end{equation}
The image of the rectangle $R_3$ under $T_{\alpha,\beta}$ splits along the edge $E_3$ into congruent halves.
The image of~$R_3^-$, the left half of~$R_3$, under $T_{\alpha,\beta}^2$ forms the left half of~$R^*_1$.
The image of~$R_3^+$, the right half of~$R_3$, under $T_{\alpha,\beta}^2$ forms the right half of~$R^*_2$.
Thus
\begin{equation}\label{eq1.14}
\CCC_3=\CCC^*_1=\CCC^*_2.
\end{equation}
Combining \eqref{eq1.12}--\eqref{eq1.14}, we conclude $\CCC_1=\CCC_2=\CCC_3$.
This means that the dominant colour in the rectangles $R_1,R_2,R_3$ is the same colour~$\CCC$.

Thus we can find measurable subsets $R_\sigma^{(0)}\subset R_\sigma$, $\sigma=1,2,3$, satisfying
\begin{displaymath}
R_\sigma^{(0)}\cap\CCC=R_\sigma^{(0)}
\quad\mbox{and}\quad
\frac{\lambda_2(R_\sigma^{(0)})}{\lambda_2(R_\sigma)}=\frac{9}{10},
\end{displaymath}
and such that their images on $[0,1)^2$ under the projection \eqref{eq1.3} coincide.
This clearly contradicts \eqref{eq1.9}, and establishes the ergodicity of $T_{\alpha,\beta}$.

The Birkhoff ergodic theorem now implies that for any fixed Kronecker direction $\bfv\in\Rr^3$, a half-infinite geodesic
in the L-solid translation $3$-manifold with almost any starting point is uniformly distributed.
This establishes a weaker form of Theorem~\ref{thm1}.
In Section~\ref{sec4}, we extend ergodicity to unique ergodicity and complete the proof.

Theorem~\ref{thm1} is best possible in the sense that there is no uniform distribution if the direction $\bfv=(v_1,v_2,v_3)\in\Rr^3$
of a half-infinite geodesic is not a Kronecker direction, so that the coordinates $v_1,v_2,v_3$ are linearly dependent over~$\Qq$.
In this case, we do not even have density.
Indeed, the classical Kronecker density theorem implies that even the \textit{modulo one} projection of the geodesic orbit
to the unit torus $[0,1)^3$ is not dense.

%
%

\section{Splitting diagram and splitting permutation}\label{sec2}

The key part in the proof of ergodicity in the case of geodesics in the L-solid translation $3$-manifold
is illustrated by the \textit{splitting diagram} given in Figure~2.1, which is a simpler version of Figure~1.6.

\begin{displaymath}
\begin{array}{c}
\includegraphics[scale=0.8]{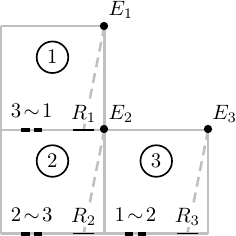}
\vspace{3pt}\\
\mbox{Figure 2.1: splitting diagram of the L-solid translation $3$-manifold}
\end{array}
\end{displaymath}

Here we label the top, middle and right cubes by the numbers $1,2,3$ respectively.
Then for every $\sigma=1,2,3$, $Y_\sigma$ is the bottom $Y$-face of cube~$\sigma$,
$R_\sigma$ is a small special rectangle on~$Y_\sigma$,
and the image of $R_\sigma$ under $T_{\alpha,\beta}$ splits along the edge~$E_\sigma$,
and we then study the images of the congruent halves $R_\sigma^-$ and $R_\sigma^+$ of $R_\sigma$ under $T_{\alpha,\beta}^2$.
Furthermore, for every $\sigma=1,2,3$, the small special rectangle~$R_\sigma$, and so also its images under $T_{\alpha,\beta}^2$,
has predominantly the colour~$\CCC_\sigma$.

The label $3\sim1$ on the face $Y_1$ indicates that $T_{\alpha,\beta}^2(R_3^-)\cup T_{\alpha,\beta}^2(R_1^+)$
is, apart from a set of zero measure, a rectangle on the face~$Y_1$, and $\CCC_3=\CCC_1$.
The label $2\sim3$ on the face $Y_2$ indicates that $T_{\alpha,\beta}^2(R_2^-)\cup T_{\alpha,\beta}^2(R_3^+)$
is, apart from a set of zero measure, a rectangle on the face~$Y_2$, and $\CCC_2=\CCC_3$.
The label $1\sim2$ on the face $Y_3$ indicates that $T_{\alpha,\beta}^2(R_1^-)\cup T_{\alpha,\beta}^2(R_2^+)$
is, apart from a set of zero measure, a rectangle on the face~$Y_3$, and $\CCC_1=\CCC_2$.
Furthermore, we can compress this information into a single cyclic permutation $3\to1\to2\to3$.
We call this the \textit{splitting permutation} of the L-solid translation $3$-manifold.

Our argument in Section~\ref{sec1} can therefore be adapted to study geodesic flow on a class of finite polycube translation
$3$-manifolds where each member is essentially the cartesian product of a finite polysquare translation surface and the
unit torus $[0,1)$, provided that the corresponding splitting permutation is a single cyclic permutation.

We now examine a few more examples.

\begin{example}\label{ex21}
Consider the $5$-cube staircase translation $3$-manifold, with splitting diagram given in Figure~2.2.

\begin{displaymath}
\begin{array}{c}
\includegraphics[scale=0.8]{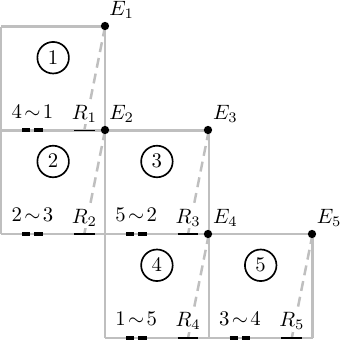}
\vspace{3pt}\\
\mbox{Figure 2.2: splitting diagram of the $5$-cube staircase translation $3$-manifold}
\end{array}
\end{displaymath}

The splitting permutation
\begin{equation}\label{eq2.1}
4\to1\to5\to2\to3\to4
\end{equation}
is given in the form of a single cyclic permutation.

Let us add $2$ cubes to this $3$-manifold and consider a $7$-cube staircase translation $3$-manifold,
with splitting diagram given in Figure~2.3.

\begin{displaymath}
\begin{array}{c}
\includegraphics[scale=0.8]{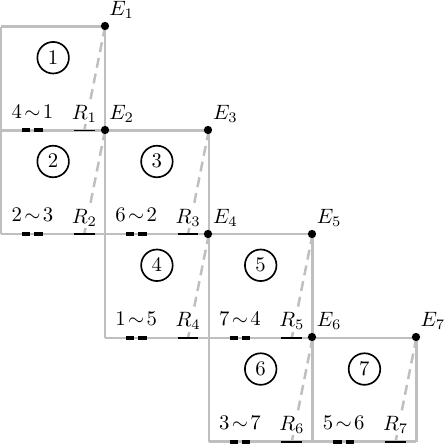}
\vspace{3pt}\\
\mbox{Figure 2.3: splitting diagram of the $7$-cube staircase translation $3$-manifold}
\end{array}
\end{displaymath}

The splitting permutation
\begin{equation}\label{eq2.2}
4\to1\to5\to6\to2\to3\to7\to4
\end{equation}
is also given in the form of a single cyclic permutation.

Note that to go from \eqref{eq2.1} to \eqref{eq2.2}, we insert $7$ after~$3$, and insert $6$ after~$5$.
The insertion corresponds to the following observations after adding cubes $6$ and~$7$:

$\circ$ The image $T_{\alpha,\beta}^2(R_3^-)$ moves from $Y_5$ to~$Y_6$.

$\circ$ The image $T_{\alpha,\beta}^2(R_5^-)$ moves from $Y_3$ to~$Y_7$.

$\circ$ The image $T_{\alpha,\beta}^2(R_\sigma^-)$ is unchanged for $\sigma=1,2$ and $\sigma=4$.

$\circ$ The image $T_{\alpha,\beta}^2(R_\sigma^+)$ is unchanged for $\sigma=1,2,3,4,5$.

$\circ$ The images $T_{\alpha,\beta}^2(R_6^-)$ and $T_{\alpha,\beta}^2(R_7^-)$ are on $Y_3$ and $Y_5$ respectively.

$\circ$ The images $T_{\alpha,\beta}^2(R_6^+)$ and $T_{\alpha,\beta}^2(R_7^+)$ are on $Y_7$ and $Y_6$ respectively.

It can be shown that if we go from the $s$-cube staircase translation $3$-manifold to the $(s+2)$-cube staircase
translation $3$-manifold, where $s$ is odd, then the following assertions are valid:

$\circ$ The image $T_{\alpha,\beta}^2(R_{s-2}^-)$ moves from $Y_s$ to~$Y_{s+1}$.

$\circ$ The image $T_{\alpha,\beta}^2(R_s^-)$ moves from $Y_{s-2}$ to~$Y_{s+2}$.

$\circ$ The image $T_{\alpha,\beta}^2(R_\sigma^-)$ is unchanged for $\sigma=1,\ldots,s-3$ and $\sigma=s-1$.

$\circ$ The image $T_{\alpha,\beta}^2(R_\sigma^+)$ is unchanged for $\sigma=1,\ldots,s$.

$\circ$ The images $T_{\alpha,\beta}^2(R_{s+1}^-)$ and $T_{\alpha,\beta}^2(R_{s+2}^-)$ are on $Y_{s-2}$ and $Y_s$ respectively.

$\circ$ The images $T_{\alpha,\beta}^2(R_{s+1}^+)$ and $T_{\alpha,\beta}^2(R_{s+2}^+)$ are on $Y_{s+2}$ and $Y_{s+1}$ respectively.

This leads to the insertion of $s+2$ after $s-2$ and the insertion of $s+1$ after $s$ in the splitting permutation.
Indeed, this argument repeatedly inductively then shows that for every $s$-cube staircase translation $3$-manifold, where $s$ is odd,
the splitting permutation is a single cyclic permutation.
\end{example}

\begin{example}\label{ex22}
A $(3,2)$-snake translation surface is a finite polysquare translation surface where the square faces form a finite zigzagging path
from left to right, every street has length~$3$, apart from the two streets at the end which have length~$2$.
A $(3,2)$-snake translation $3$-manifold is then the cartesian product of a $(3,2)$-snake translation surface and the unit torus $[0,1)$.
Figure~2.4 is the splitting diagram of such a $3$-manifold with $7$ cubes.

\begin{displaymath}
\begin{array}{c}
\includegraphics[scale=0.8]{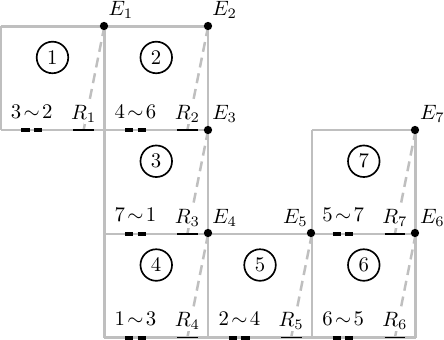}
\vspace{3pt}\\
\mbox{Figure 2.4: splitting diagram of a $7$-cube $(3,2)$-snake translation $3$-manifold}
\end{array}
\end{displaymath}

The splitting permutation
\begin{equation}\label{eq2.3}
3\to2\to4\to6\to5\to7\to1\to3
\end{equation}
is given in the form of a single cyclic permutation.

Let us add $2$ cubes to this $3$-manifold and consider a $9$-cube $(3,2)$-snake translation $3$-manifold,
with splitting diagram given in Figure~2.5.

\begin{displaymath}
\begin{array}{c}
\includegraphics[scale=0.8]{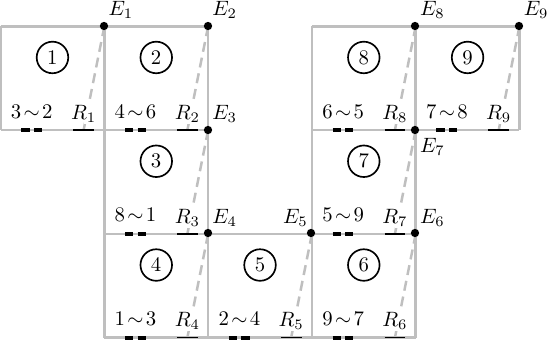}
\vspace{3pt}\\
\mbox{Figure 2.5: splitting diagram of a $9$-cube $(3,2)$-snake translation $3$-manifold}
\end{array}
\end{displaymath}

The splitting permutation
\begin{equation}\label{eq2.4}
3\to2\to4\to6\to5\to9\to7\to8\to1\to3
\end{equation}
is also given in the form of a single cyclic permutation.

Note that to go from \eqref{eq2.3} to \eqref{eq2.4}, we insert $9$ after~$5$, and insert $8$ after~$7$.
Here the general case is more complicated.
As we go from left to right, any vertical street along the way may go up or down, and the different choices lead to somewhat
different splitting diagrams.
\end{example}

\begin{example}\label{ex23}
Figure~2.6 shows the splitting diagram of a polycube translation $3$-manifold which is the cartesian product of a
finite polysquare translation surface with a hole and the unit torus $[0,1)$.
Here the finite polysquare translation surface has edge identification via perpendicular translation, and
has $5$ horizontal and $6$ vertical streets.

\begin{displaymath}
\begin{array}{c}
\includegraphics[scale=0.8]{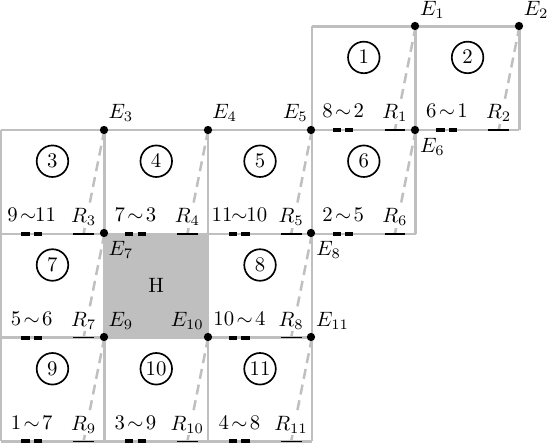}
\vspace{3pt}\\
\mbox{Figure 2.6: splitting diagram of a finite polycube translation $3$-manifold}
\\
\mbox{with a hole}
\end{array}
\end{displaymath}

The splitting permutation
\begin{displaymath}
8\to2\to5\to6\to1\to7\to3\to9\to11\to10\to4\to8
\end{displaymath}
is given in the form of a single cyclic permutation.
\end{example}

The next two examples concern cases where the splitting permutation is not a single cyclic permutation.
This can fail in more than one way.

\begin{example}\label{ex24}
Here we consider a very simple case, where we add an extra cube to the L-solid translation $3$-manifold to form
a $4$-cube L-solid translation $3$-manifold.
Figure~2.7 shows the splitting diagram.

\begin{displaymath}
\begin{array}{c}
\includegraphics[scale=0.8]{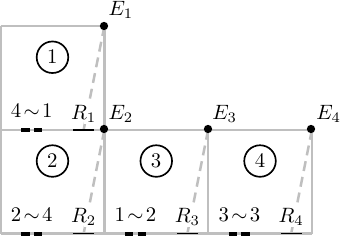}
\vspace{3pt}\\
\mbox{Figure 2.7: splitting diagram of a $4$-cube L-solid translation $3$-manifold}
\end{array}
\end{displaymath}

The splitting permutation is given by $4\to1\to2\to4$ and $3\to3$ which is clearly not a single cyclic permutation.
There is a stationary point~$3$.
\end{example}

\begin{example}\label{ex25}
Consider a polycube translation $3$-manifold which is the cartesian product of a finite polysquare translation surface
with $6$ square faces and the unit torus $[0,1)$.
Figure~2.8 shows the splitting diagram.

\begin{displaymath}
\begin{array}{c}
\includegraphics[scale=0.8]{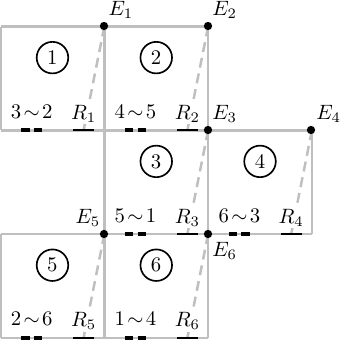}
\vspace{3pt}\\
\mbox{Figure 2.8: splitting diagram of a $6$-cube solid polysquare translation $3$-manifold}
\end{array}
\end{displaymath}

Here the splitting permutation
\begin{displaymath}
3\to2\to6\to3
\quad\mbox{and}\quad
4\to5\to1\to4
\end{displaymath}
is a product of two disjoint cyclic permutations.
\end{example}

For the class of polycube translation $3$-manifolds which are cartesian products
of finite polysquare translation surfaces and the unit torus $[0,1)$,
the question of which of these have splitting permutations which are cyclic has a rather simple answer.
Cyclic splitting permutation means a single cycle in the cycle decomposition, and this happens if and only if
the boundary edge identification of the finite polysquare translation surface reduces the number of vertices to~$1$.
Indeed, this reduction to $1$ vertex is precisely the same process by which the colours in the splitting points
are proved to be equal and we end up with a single colour.
And \textit{vice versa}.

In Section~\ref{sec5}, we prove the following result.

\begin{theorem}\label{thm2}
Suppose that a polycube translation $3$-manifold $\MMM$ is the cartesian product of a polysquare translation surface $\PPP$
with $d$ square faces and the unit torus $[0,1)$.
Suppose further that the splitting permutation contains a cycle of length greater than~$d/2$.
Then any half-infinite geodesic in $\MMM$ with a Kronecker direction $\bfv\in\Rr^3$ is uniformly distributed unless it hits
a singularity.
\end{theorem}

It is clear that the restriction that we have imposed on the polycube translation $3$-manifolds is a very severe one.
It is an interesting open problem to study the corresponding problem for general finite polycube translation $3$-manifolds.

%
%

\section{Proof of Lemma~\ref{lem12}}\label{sec3}

Recall that we denote the Kronecker direction by $\bfv=(\alpha,1,\beta)$, where $\alpha$, $\beta$ and $\alpha/\beta$ are all irrational,
and also assume that $0<\alpha<1$ and $\beta>0$.

The irrational number $\alpha$ has infinite continued fraction
\begin{equation}\label{eq3.1}
\alpha=[0;a_1,a_2,a_3,\ldots]=\frac{1}{a_1+\frac{1}{a_2+\frac{1}{a_3+\ldots}}},
\end{equation}
where $a_1,a_2,a_3,\ldots$ are positive integers.
For any $k=1,2,3,\ldots,$ the $k$-th convergent
\begin{displaymath}
\frac{p_k}{q_k}=\frac{p_k(\alpha)}{q_k(\alpha)}=[0;a_1,a_2,a_3,\ldots,a_k],
\end{displaymath}
where the integers $p_k$ and $q_k$ are coprime, is defined in terms of a finite initial segment of the continued fraction \eqref{eq3.1},
with denominator $q_k=q_k(\alpha)$.
It is well known that
\begin{displaymath}
\Vert q\alpha\Vert\ge\Vert q_k\alpha\Vert=\vert q_k\alpha-p_k\vert,
\quad
q=1,\ldots,q_{k+1}-1,
\end{displaymath}
\begin{equation}\label{eq3.2}
\Vert q_{k+1}\alpha\Vert<\Vert q_k\alpha\Vert,
\end{equation}
and
\textcolor{white}{xxxxxxxxxxxxxxxxxxxxxxxxxxxxxx}
\begin{displaymath}
\frac{1}{q_{k+1}+q_k}\le\Vert q_k\alpha\Vert\le\frac{1}{q_{k+1}}.
\end{displaymath}
These show that the convergents give the best rational approximations to~$\alpha$.

For any integer $k\ge1$, let $\AAA_k(\alpha)$ be the partition of $[0,1)$ with $q_{k+1}$ division points
\begin{equation}\label{eq3.3}
\{\ell\alpha\},
\quad
\ell=0,1,2,\ldots,q_{k+1}-1.
\end{equation}
We need information on the gaps between neighboring partition points, and these can be described in terms of continued fractions.
We also use the famous $3$-distance theorem \cite{halton65,slater67,sos57,sos58,suranyi58,swierczkowski59} in its strong form.

\begin{lemma}[$2$-distance theorem]\label{lem31}
The distance between any pair of neighboring partition points \eqref{eq3.3} of $\AAA_k(\alpha)$ can only take the values
\begin{displaymath}
\Vert q_k\alpha\Vert
\quad\mbox{and}\quad
\Vert q_k\alpha\Vert+\Vert q_{k+1}\alpha\Vert,
\end{displaymath}
precisely $q_{k+1}-q_k$ and $q_k$ times respectively.
\end{lemma}

For every $i=1,\ldots,q_{k+1}-1$, consider the short special interval
\begin{equation}\label{eq3.4}
J_k(i)=J_k(\alpha;i)=\left(\{i\alpha\}-\frac{\Vert q_k\alpha\Vert}{2},\{i\alpha\}+\frac{\Vert q_k\alpha\Vert}{2}\right),
\end{equation}
with length
\textcolor{white}{xxxxxxxxxxxxxxxxxxxxxxxxxxxxxx}
\begin{displaymath}
\vert J_k(\alpha;i)\vert=\Vert q_k\alpha\Vert.
\end{displaymath}
In view of $\Vert q_k\alpha\Vert+\Vert q_{k+1}\alpha\Vert<2\Vert q_k\alpha\Vert$, a simple consequence of \eqref{eq3.2},
it is clear that these intervals (i) are pairwise disjoint, (ii) are contained in the open interval $(0,1)$,
and (iii) have total length greater than~$1/3$.

The irrational number $\beta$ has infinite continued fraction
\begin{equation}\label{eq3.5}
\beta=[b_0;b_1,b_2,b_3,\ldots]=b_0+\frac{1}{b_1+\frac{1}{b_2+\frac{1}{b_3+\ldots}}},
\end{equation}
where $b_0$ is a non-negative integer and $b_1,b_2,b_3,\ldots$ are positive integers.
For any $h=1,2,3,\ldots,$ the $h$-th convergent
\begin{displaymath}
\frac{p'_h}{q'_h}=\frac{p'_h(\beta)}{q'_h(\beta)}=[b_0;b_1,b_2,b_3,\ldots,b_h],
\end{displaymath}
where the integers $p'_h$ and $q'_h$ are coprime, is defined in terms of a finite initial segment of the continued fraction \eqref{eq3.5},
with denominator $q'_h=q'_h(\beta)$.

For any integer $h\ge1$, let $\BBB_h(\beta)$ be the partition of $[0,1)$ with $q'_{h+1}$ division points
\begin{equation}\label{eq3.6}
\{\ell\beta\},
\quad
\ell=0,1,2,\ldots,q'_{h+1}-1.
\end{equation}
In view of Lemma~\ref{lem31}, the distance between any pair of neighboring partition points \eqref{eq3.6} of $\BBB_h(\beta)$
can only take the values
\begin{equation}\label{eq3.7}
\Vert q'_h\beta\Vert
\quad\mbox{and}\quad
\Vert q'_h\beta\Vert+\Vert q'_{h+1}\beta\Vert.
\end{equation}

For every integer $j=1,\ldots,q'_{h+1}-1$, consider the short special interval
\begin{equation}\label{eq3.8}
J'_h(j)=J'_h(\beta;j)=\left(\{j\beta\}-\frac{\Vert q'_h\beta\Vert}{2},\{j\beta\}+\frac{\Vert q'_h\beta\Vert}{2}\right),
\end{equation}
with length
\textcolor{white}{xxxxxxxxxxxxxxxxxxxxxxxxxxxxxx}
\begin{displaymath}
\vert J'_h(\beta;j)\vert=\Vert q'_h\beta\Vert.
\end{displaymath}
It is clear that these intervals (i) are pairwise disjoint, (ii) are contained in the open interval $(0,1)$,
and (iii) have total length greater than~$1/3$.

Using the short special intervals above, we construct some small special rectangles.

Let the denominator $q_k=q_k(\alpha)$ of the $k$-th convergent of $\alpha$ be \textit{large}, and let the denominator
$q'_h=q'_h(\beta)$ of the $h$-th convergent of $\beta$ be \textit{substantially larger}, both to be made precise later.

The faces $Y_1,Y_2,Y_3$ of the L-solid translation $3$-manifold have bottom left vertex
\begin{displaymath}
\bfp_1=(0,1,0),
\quad
\bfp_2=(0,0,0),
\quad
\bfp_3=(1,0,0)
\end{displaymath}
respectively.
For every $\sigma=1,2,3$ and every $j=1,\ldots,q'_{h+1}-q_{k+1}$, we consider the small special rectangle
\begin{displaymath}
R_{\sigma,k,h}(\alpha;1;\beta;j)=\{\bfp_\sigma+(x,0,z)\in Y_\sigma:(x,z)\in J_k(\alpha;1)\times J'_h(\beta;j)\}
\end{displaymath}
on the face~$Y_\sigma$.
For every such $\sigma$ and~$j$, and for every $i=0,1,\ldots,q_{k+1}-2$, we also consider the image of this small special rectangle
under $T_{\alpha,\beta}^i$, given by
\begin{equation}\label{eq3.9}
R^\dagger_{\sigma,k,h}(\alpha;1+i;\beta;j+i)=T_{\alpha,\beta}^i(R_{\sigma,k,h}(\alpha;1;\beta;j))
\end{equation}
which is not necessarily in~$Y_\sigma$ but can be anywhere in~$\YYY$, and is \textit{splitting free}, in view of the properties
(i)--(iii) of the short special intervals.
In particular, the lack of splitting implies that they are congruent rectangles, with the same area
\begin{equation}\label{eq3.10}
\lambda_2(R^\dagger_{\sigma,k,h}(\alpha;1+i;\beta;j+i))=\lambda_1(J_k(\alpha;1))\lambda_1(J'_h(\beta;1)),
\end{equation}
where $\lambda_1$ and $\lambda_2$ denote respectively $1$- and $2$-dimensional Lebesgue measures.
Thus they are also small special rectangles.

Next, recall from Section~\ref{sec1} and \eqref{eq1.9} that the underlying set $\YYY$ has a non-trivial decomposition
into a disjoint union $\YYY=\WWW\cup\SSS$ of $T_{\alpha,\beta}$-invariant measurable subsets $\WWW$ and~$\SSS$,
with $\lambda_2(\WWW)=1$ and $\lambda_2(\SSS)=2$.
Let $\eps>0$ be arbitrarily small but fixed.
Then there exists a subset $\WWW_1=\WWW_1(\eps;\WWW)\subset\YYY$ which is a finite union of disjoint rectangles
and such that the symmetric difference $\WWW\bigtriangleup\WWW_1=(\WWW\setminus\WWW_1)\cup(\WWW_1\setminus\WWW)$
has measure
\textcolor{white}{xxxxxxxxxxxxxxxxxxxxxxxxxxxxxx}
\begin{equation}\label{eq3.11}
\lambda_2(\WWW\bigtriangleup\WWW_1)<\eps.
\end{equation}
Note also that the $\YYY$-complement $\WWW_1^{\,c}=\YYY\setminus\WWW_1$ of $\WWW_1$ satisfies
\begin{equation}\label{eq3.12}
\WWW\bigtriangleup\WWW_1=\SSS\bigtriangleup\WWW_1^{\,c}.
\end{equation}

For any pair $(\sigma,j)$ with $\sigma=1,2,3$ and $j=1,\ldots,q'_{h+1}-q_{k+1}$, we call the collection
\begin{equation}\label{eq3.13}
R^\dagger_{\sigma,k,h}(\alpha;1+i;\beta;j+i),
\quad
i=0,1,\ldots,q_{k+1}-2,
\end{equation}
of small special rectangles the $T_{\alpha,\beta}$-power chain of the pair $(\sigma,j)$.
Let $\delta=\delta(\eps)>0$ be fixed.
We say that this $T_{\alpha,\beta}$-power chain \eqref{eq3.13} is \textit{defective} if every member of the chain is \textit{bad}
in the sense that at least one of the following two properties holds:

(1)
The member intersects the boundary of a rectangle of the set~$\WWW_1$.

(2)
The inequality
\begin{equation}\label{eq3.14}
\frac{\lambda_2(R^\dagger_{\sigma,k,h}(\alpha;1+i;\beta;j+i)\cap(\WWW\bigtriangleup\WWW_1))}
{\lambda_2(R^\dagger_{\sigma,k,h}(\alpha;1+i;\beta;j+i))}
\ge\delta
\end{equation}
holds.

To show that defective $T_{\alpha,\beta}$-power chains form a small minority, 
let $\Delta$ denote the total number of pairs $(\sigma,j)$ satisfying $\sigma=1,2,3$ and $j=1,\ldots,q'_{h+1}-q_{k+1}$
such that the $T_{\alpha,\beta}$-power chain \eqref{eq3.13} is defective.
Then the total number $\Delta(q_{k+1}-1)$ of small special rectangles \eqref{eq3.9} with
\begin{equation}\label{eq3.15}
\sigma=1,2,3,
\quad
i=0,1,\ldots,q_{k+1}-2,
\quad
j=1,\ldots,q'_{h+1}-q_{k+1}
\end{equation}
in defective $T_{\alpha,\beta}$-power chains satisfies
\begin{equation}\label{eq3.16}
\Delta(q_{k+1}-1)\le\Lambda+\Omega,
\end{equation}
where $\Lambda$ denotes the total number of small special rectangles \eqref{eq3.9} with \eqref{eq3.15}
that intersect the boundary of a rectangle in the set~$\WWW_1$,
and $\Omega$ denotes the total number of small special rectangles \eqref{eq3.9} with \eqref{eq3.15}
such that the inequality \eqref{eq3.14} holds.

Consider a face $Y_\sigma$, where a typical point is of the form $\bfp_\sigma+(x,0,z)$.
Note that the intervals $J_k(\alpha;1+i)$, $i=0,1,\ldots,q_{k+1}-2$, are disjoint, and for each fixed~$i$,
the intervals $J'_h(\beta;j+i)$, $j=1,\ldots,q'_{h+1}-q_{k+1}$, are disjoint.
It follows that on~$Y_\sigma$, any line segment in the $x$-direction intersects less than $q_{k+1}$ of the small special rectangles,
while any line segment in the $z$-direction intersects less than $q'_{h+1}$ of the small special rectangles.
This means that the boundary of any rectangle in the set $\WWW_1$ intersects less than $2(q_{k+1}+q'_{h+1})$ small special rectangles.
Since $\WWW_1$ contains only finitely many rectangles, it follows that there exists a constant $c_1=c_1(\eps;\WWW)$
such that
\textcolor{white}{xxxxxxxxxxxxxxxxxxxxxxxxxxxxxx}
\begin{equation}\label{eq3.17}
\Lambda\le c_1(q_{k+1}+q'_{h+1}).
\end{equation}

Next, note that properties (i)--(iii) imply the trivial bounds
\begin{displaymath}
\lambda_1(J_k(\alpha;1))\ge\frac{1}{3(q_{k+1}-1)}
\quad\mbox{and}\quad
\lambda_1(J'_k(\beta;1))\ge\frac{1}{3(q'_{h+1}-1)},
\end{displaymath}
so that
\textcolor{white}{xxxxxxxxxxxxxxxxxxxxxxxxxxxxxx}
\begin{equation}\label{eq3.18}
\lambda_2(R^\dagger_{\sigma,k,h}(\alpha;1+i;\beta;j+i))
\ge\frac{1}{9(q_{k+1}-1)(q'_{h+1}-1)}.
\end{equation}
Combining \eqref{eq3.11}, \eqref{eq3.14} and \eqref{eq3.18}, we deduce that
\begin{displaymath}
\eps>\lambda_2(\WWW\bigtriangleup\WWW_1)\ge\frac{\Omega\delta}{9(q_{k+1}-1)(q'_{h+1}-1)}.
\end{displaymath}
This inequality can be simplified to
\begin{equation}\label{eq3.19}
\Omega<\eps^{1/2}(q_{k+1}-1)(q'_{h+1}-1)
\end{equation}
if we choose the parameter $\delta$ to satisfy
\begin{equation}\label{eq3.20}
\delta=\delta(\eps)=9\eps^{1/2}.
\end{equation}

Combining \eqref{eq3.16}, \eqref{eq3.17} and \eqref{eq3.19}, we deduce that
\begin{displaymath}
\Delta<\frac{c_1(q_{k+1}+q'_{h+1})}{q_{k+1}-1}+\eps^{1/2}(q'_{h+1}-1).
\end{displaymath}
Since there are precisely $3(q'_{h+1}-q_{k+1})$ $T_{\alpha,\beta}$-power chains, it then follows that
as long as \eqref{eq3.20} holds, the proportion of defective $T_{\alpha,\beta}$-power chains is at most
\begin{equation}\label{eq3.21}
\frac{c_1(q_{k+1}+q'_{h+1})}{3(q'_{h+1}-q_{k+1})(q_{k+1}-1)}+\frac{\eps^{1/2}(q'_{h+1}-1)}{3(q'_{h+1}-q_{k+1})},
\end{equation}

We now repeat the argument for the reverse flow.

For any integer $k\ge1$, we consider the partition $\AAA_k(-\alpha)$ of $[0,1)$ with $q_{k+1}$ division points
\textcolor{white}{xxxxxxxxxxxxxxxxxxxxxxxxxxxxxx}
\begin{displaymath}
\{-\ell\alpha\},
\quad
\ell=0,1,2,\ldots,q_{k+1}-1.
\end{displaymath}
For every $i=1,\ldots,q_{k+1}-1$, we consider the special short interval $J_k(\alpha;-i)$, defined by replacing $i$ by $-i$
in \eqref{eq3.4}.

For every $\sigma=1,2,3$ and every $j=q_{k+1}+1,\ldots,q'_{h+1}$, we consider the small special rectangle
\begin{displaymath}
R_{\sigma,k,h}(\alpha;-1;\beta;j-2)
=\{\bfp_\sigma+(x,0,z)\in Y_\sigma:(x,z)\in J_k(\alpha;-1)\times J'_h(\beta;j-2)\}
\end{displaymath}
on the face~$Y_\sigma$.
For every $\sigma$ and~$j$, and for every $i=0,1,\ldots,q_{k+1}-2$, we also consider the image of this small special rectangle
under $T_{\alpha,\beta}^{-i}$, given by
\begin{displaymath}
R^\dagger_{\sigma,k,h}(\alpha;-1-i;\beta;j-2-i)=T_{\alpha,\beta}^{-i}(R_{\sigma,k,h}(\alpha;-1;\beta;j-2))
\end{displaymath}
which is not necessarily in~$Y_\sigma$ but can be anywhere in~$\YYY$, and is splitting free.
Thus these are also small special rectangles, with the same area given by \eqref{eq3.10}.

As before, for any pair $(\sigma,j)$ with $\sigma=1,2,3$ and $j=q_{k+1}+1,\ldots,q'_{h+1}$, we say that
the $T_{\alpha,\beta}^{-1}$-power chain of the pair $(\sigma,j)$, given by
\begin{equation}\label{eq3.22}
R^\dagger_{\sigma,k,h}(\alpha;-1-i;\beta;j-2-i),
\quad
i=0,1,\ldots,q_{k+1}-2,
\end{equation}
is defective if every member of the chain is bad in the sense that at least one of the following two properties holds:

(1)
The member intersects the boundary of a rectangle of the set~$\WWW_1$.

(2)
The inequality
\begin{displaymath}
\frac{\lambda_2(R^\dagger_{\sigma,k,h}(\alpha;-1-i;\beta;j-2-i)\cap(\WWW\bigtriangleup\WWW_1))}
{\lambda_2(R^\dagger_{\sigma,k,h}(\alpha;-1-i;\beta;j-2-i))}
\ge\delta
\end{displaymath}
holds, where $\delta$ is defined as before.

An analogous argument then shows that the proportion of defective $T_{\alpha,\beta}^{-1}$-power chains is at most \eqref{eq3.21}.

\begin{lemma}\label{lem32}
Suppose that the proportion \eqref{eq3.21} both of defective $T_{\alpha,\beta}$-power chains \eqref{eq3.13}
and of defective $T_{\alpha,\beta}^{-1}$-power chains \eqref{eq3.22} is sufficiently small.
Then there exists an integer $j_0$ satisfying $q_{k+1}+1\le j_0\le q'_{h+1}-q_{k+1}$
such that for every $\sigma=1,2,3$,

\emph{(i)}
the $T_{\alpha,\beta}$-power chain \eqref{eq3.13} with $j=j_0$ is not defective; and

\emph{(ii)}
the $T_{\alpha,\beta}^{-1}$-power chain \eqref{eq3.22} with $j=j_0$ is not defective.
\end{lemma}

It follows from (i) that for every $\sigma=1,2,3$, there exists an integer $i_0$ satisfying $0\le i_0\le q_{k+1}-2$
such that the small special rectangle
\begin{displaymath}
R^\dagger_{\sigma,k,h}(\alpha;1+i_0;\beta;j_0+i_0)
\end{displaymath}
does not intersect the boundary of a rectangle of the set $\WWW_1$ and also satisfies the inequality
\textcolor{white}{xxxxxxxxxxxxxxxxxxxxxxxxxxxxxx}
\begin{equation}\label{eq3.23}
\frac{\lambda_2(R^\dagger_{\sigma,k,h}(\alpha;1+i_0;\beta;j_0+i_0)\cap(\WWW\bigtriangleup\WWW_1))}
{\lambda_2(R^\dagger_{\sigma,k,h}(\alpha;1+i_0;\beta;j_0+i_0))}
<\delta.
\end{equation}
The former implies that
\begin{equation}\label{eq3.24}
R^\dagger_{\sigma,k,h}(\alpha;1+i_0;\beta;j_0+i_0)\subset\WWW_1
\quad\mbox{or}\quad
R^\dagger_{\sigma,k,h}(\alpha;1+i_0;\beta;j_0+i_0)\subset\WWW_1^c.
\end{equation}
If the first condition in \eqref{eq3.24} holds, then it follows from \eqref{eq3.23} that
\begin{equation}\label{eq3.25}
\frac{\lambda_2(R^\dagger_{\sigma,k,h}(\alpha;1+i_0;\beta;j_0+i_0)\cap\WWW)}
{\lambda_2(R^\dagger_{\sigma,k,h}(\alpha;1+i_0;\beta;j_0+i_0))}
\ge1-\delta.
\end{equation}
If the second condition in \eqref{eq3.24} holds, then it follows from \eqref{eq3.12} and \eqref{eq3.23} that
\begin{equation}\label{eq3.26}
\frac{\lambda_2(R^\dagger_{\sigma,k,h}(\alpha;1+i_0;\beta;j_0+i_0)\cap\SSS)}
{\lambda_2(R^\dagger_{\sigma,k,h}(\alpha;1+i_0;\beta;j_0+i_0))}
\ge1-\delta.
\end{equation}
Since $T_{\alpha,\beta}$ is measure-preserving and the subsets $\WWW$ and $\SSS$ are $T_{\alpha,\beta}$-invariant,
it follows that each condition \eqref{eq3.25} or \eqref{eq3.26}, if true, extends to the whole $T_{\alpha,\beta}$-power chain.
Thus in particular, one of the inequalities
\begin{displaymath}
\frac{\lambda_2(R^\dagger_{\sigma,k,h}(\alpha;1;\beta;j_0)\cap\WWW)}
{\lambda_2(R^\dagger_{\sigma,k,h}(\alpha;1;\beta;j_0))}
\ge1-\delta
\quad\mbox{or}\quad
\frac{\lambda_2(R^\dagger_{\sigma,k,h}(\alpha;1;\beta;j_0)\cap\SSS)}
{\lambda_2(R^\dagger_{\sigma,k,h}(\alpha;1;\beta;j_0))}
\ge1-\delta
\end{displaymath}
holds.
Similarly, it follows from (ii) that one of the inequalities
\begin{displaymath}
\frac{\lambda_2(R^\dagger_{\sigma,k,h}(\alpha;-1;\beta;j_0-2)\cap\WWW)}
{\lambda_2(R^\dagger_{\sigma,k,h}(\alpha;-1;\beta;j_0-2))}
\ge1-\delta
\quad\!\mbox{or}\quad\!
\frac{\lambda_2(R^\dagger_{\sigma,k,h}(\alpha;-1;\beta;j_0-2)\cap\SSS)}
{\lambda_2(R^\dagger_{\sigma,k,h}(\alpha;-1;\beta;j_0-2))}
\ge1-\delta
\end{displaymath}
holds.
For $\sigma=1,2,3$, we now take
\begin{displaymath}
R_\sigma
=R^\dagger_{\sigma,k,h}(\alpha;-1;\beta;j_0-2)
=R_{\sigma,k,h}(\alpha;-1;\beta;j_0-2)
\end{displaymath}
on the face~$Y_\sigma$, and take
\begin{displaymath}
R^*_\sigma
=R^\dagger_{\sigma,k,h}(\alpha;1;\beta;j_0)
=R_{\sigma,k,h}(\alpha;1;\beta;j_0)
\end{displaymath}
on the face~$Y_\sigma$.

To complete the proof of Lemma~\ref{lem12}, it remains to specify the parameters.

To satisfy the condition \eqref{eq1.11} and its analogues, we need $1-\delta>99/100$, so that $\delta<1/100$.
In view of \eqref{eq3.20}, we need the parameter $\eps>0$ to satisfy
\begin{displaymath}
\eps<\frac{1}{810000}.
\end{displaymath}
The parameter $\eps>0$ also needs to be sufficiently small to ensure that the proportion \eqref{eq3.21}
of defective power chains is small.
Note that that we can guarantee that
\begin{displaymath}
\frac{q_{k+1}+q'_{h+1}}{3(q'_{h+1}-q_{k+1})}\le1
\quad\mbox{and}\quad
\frac{q'_{h+1}-1}{3(q'_{h+1}-q_{k+1})}\le1
\end{displaymath}
by ensuring that the two continued fraction denominators $q_{k+1}$ and $q'_{h+1}$ satisfy the condition
$q'_{h+1}\ge c_2q_{k+1}$ for some suitably large absolute constant~$c_2$.
Recall next that the constant $c_1$ depends only on $\eps$ and~$\WWW$.
It follows that we can choose the continued fraction denominator $q_{k+1}$ to be sufficiently large
to ensure that the factor
\begin{displaymath}
\frac{c_1}{q_{k+1}-1}
\end{displaymath}
is as small as we please.
This completes the proof of Lemma~\ref{lem12}.

%
%

\section{Extending ergodicity to unique ergodicity}\label{sec4}

The term \textit{unique ergodicity} refers to the extension from uniformity for half-infinite geodesics with almost every starting point
to uniformity for half-infinite geodesics with every starting point and that do not hit any singularity and become undefined.
The basic ideas and pioneering results are due to Bogoliuboff and Krylov~\cite{BK37} in 1937;
see also Furstenberg \cite[Sections 3.2--3.3]{furstenberg81}.

The proof of Theorem~\ref{thm1} is by contradiction and consists of two parts.
Part~1 is a general argument based on Krylov--Bogoliuboff theory, and the key idea here is to reformulate the problem
in terms of invariant Borel measures and to use functional analysis.
Part~2 is an \textit{ad hoc} argument based on the special property that the $\bfv$-flow in the L-solid translation
$3$-manifold, when reduced \textit{modulo one} to geodesic flow in the unit torus $[0,1)^3$, exhibits uniformity.

For convenience, we denote by $\MMM$ the L-solid translation $3$-manifold.

\begin{part1}
Suppose, on the contrary, that there exist a half-infinite geodesic $\LLL(\bfv;\bfp_0;t)$, $t\ge0$, with arc-length parametization,
where $\bfv\in\Rr^3$ is a Kronecker direction and $\bfp\in\MMM$ is the starting point, and a non-zero continuous function $f_0$
on $\MMM$ for which uniformity fails.
Then the infinite sequence of integrals
\begin{displaymath}
\frac{1}{n}\int_0^nf_0(\LLL(\bfv;\bfp_0;t))\,\dd t,
\quad
n=1,2,3,\ldots,
\end{displaymath}
does not converge to
\textcolor{white}{xxxxxxxxxxxxxxxxxxxxxxxxxxxxxx}
\begin{equation}\label{eq4.1}
\frac{1}{3}\int_\MMM f_0\,\dd\lambda_3,
\end{equation}
where $\frac{1}{3}\lambda_3$ is a probability measure on~$\MMM$,
so that there exists an infinite strictly increasing sequence $h_0<h_1<h_2<h_3<\ldots$ of positive integers such that the limit
\begin{displaymath}
\lim_{n\to\infty}\frac{1}{h_n}\int_0^{h_n}f_0(\LLL(\bfv;\bfp_0;t))\,\dd t
\end{displaymath}
exists but is not equal to \eqref{eq4.1}.

\begin{lemma}\label{lem41}
Under these assumptions, there exists an ergodic measure-preserving system $(\MMM,\BBB,\nu,\mbox{$\bfv$-flow})$,
where $\BBB$ denotes the Borel $\sigma$-algebra in~$\MMM$, and $\nu$ denotes a new $\bfv$-flow-invariant Borel probability measure,
such that
\begin{equation}\label{eq4.2}
\lim_{n\to\infty}\frac{1}{h_n}\int_0^{h_n}f_0(\LLL(\bfv;\bfp_0;t))\,\dd t
=\int_\MMM f_0\,\dd\nu\ne\frac{1}{3}\int_\MMM f_0\,\dd\lambda_3.
\end{equation}
\end{lemma}

\begin{proof}
We introduce, for every integer $n=1,2,3,\ldots,$ the particular normalized
\textit{length-counting measure} $\nu_n$ on Borel sets $B$ in $\MMM$ defined by
\begin{displaymath}
\nu_n(B)=\frac{1}{h_n}\int_0^{h_n}\chi_B(\LLL(\bfv;\bfp_0;t))\,\dd t,
\end{displaymath}
where the characteristic function $\chi_B$ satisfies
\begin{displaymath}
\chi_B(\bfp)=\left\{\begin{array}{ll}
1,&\mbox{if $\bfp\in B$},\\
0,&\mbox{if $\bfp\not\in B$}.
\end{array}\right.
\end{displaymath}
We then apply a general theorem in functional analysis which says that the space of Borel probability measures
on any compact set is compact in the weak-star topology, which means that for every continuous function $f$
on the compact space, as $n\to\infty$,
\begin{equation}\label{eq4.3}
\mu_n\to\mu
\quad\mbox{if and only if}\quad
\int f\,\dd\mu_n\to\int f\,\dd\mu.
\end{equation}

\begin{remark}
This compactness theorem is a non-trivial result.
The standard proof is based on the Riesz representation theorem
and can be found in most textbooks on functional analysis.
\end{remark}

Thus the space $\frakM=\frakM(\MMM)$ of Borel probability measures on $\MMM$ is compact,
so there exists a subsequence $\nu_{n_m}$ of the sequence $\nu_n$ and a Borel probability measure $\nu_\infty$
such that $\nu_{n_m}\to\nu_\infty$ as $m\to\infty$, and it follows from \eqref{eq4.3} that
\begin{equation}\label{eq4.4}
\lim_{m\to\infty}\int_\MMM f\,\dd\nu_{n_m}=\int_\MMM f\,\dd\nu_\infty
\quad\mbox{for every continuous function $f$ on $\MMM$}.
\end{equation}

Let $t_1>0$ be arbitrary but fixed.
For any continuous function $f$ defined on~$\MMM$, let $f_{t_1}$ denote the function obtained from $f$ writing $f(t+t_1)=f_{t_1}(t)$.
Then
\begin{align}
\int_\MMM f_{t_1}\,\dd\nu_{n_m}
&
=\frac{1}{h_{n_m}}\int_0^{h_{n_m}}f_{t_1}(\LLL(\bfv;\bfp_0;t))\,\dd t
=\frac{1}{h_{n_m}}\int_0^{h_{n_m}}f(\LLL(\bfv;\bfp_0;t+t_1))\,\dd t
\nonumber
\\
&
=\frac{1}{h_{n_m}}\int_0^{h_{n_m}}f(\LLL(\bfv;\bfp_0;t))\,\dd t
\nonumber
\\
&\qquad
+\frac{1}{h_{n_m}}\int_0^{t_1}(f(\LLL(\bfv;\bfp_0;t+h_{n_m}))-f(\LLL(\bfv;\bfp_0;t)))\,\dd t
\nonumber
\\
&
=\int_\MMM f\,\dd\nu_{n_m}
+E(n_m),
\nonumber
\end{align}
where
\textcolor{white}{xxxxxxxxxxxxxxxxxxxxxxxxxxxxxx}
\begin{displaymath}
\vert E(n_m)\vert\le\frac{2t_1\sup\vert f\vert}{h_{n_m}}.
\end{displaymath}
Since $f$ is continuous on the compact set~$\MMM$, the factor $\sup\vert f\vert$ exists and is finite.
Thus $E(n_m)\to0$ as $m\to\infty$, so that
\begin{equation}\label{eq4.5}
\lim_{m\to\infty}\int_\MMM f_{t_1}\,\dd\nu_{n_m}=\lim_{m\to\infty}\int_\MMM f\,\dd\nu_{n_m}.
\end{equation}
Combining \eqref{eq4.4} and \eqref{eq4.5}, we conclude that
\begin{displaymath}
\int_\MMM f_{t_1}\,\dd\nu_\infty=\int_\MMM f\,\dd\nu_\infty,
\end{displaymath}
and this establishes the $\bfv$-flow-invariance of the limit measure $\nu_\infty$.
Furthermore, combining \eqref{eq4.2} and \eqref{eq4.4} with $f=f_0$, we conclude that
\begin{equation}\label{eq4.6}
\int_\MMM f_0\,\dd\nu_\infty\ne\frac{1}{3}\int_\MMM f_0\,\dd\lambda_3.
\end{equation}
This completes the proof.
\end{proof}

Let $\frakM_1\subset\frakM$ denote the set of Borel probability measures $\mu$ on $\MMM$ that are $\bfv$-flow-invariant
and satisfy the requirement
\begin{displaymath}
\int_\MMM f_0\,\dd\mu=\int_\MMM f_0\,\dd\nu_\infty.
\end{displaymath}
Then $\nu_\infty\in\frakM_1$, so that $\frakM_1$ is a non-empty closed subset of~$\frakM$, and so $\frakM_1$ is compact.

With some appropriate $\nu^*\in\frakM_1$, we can extend the above to a measure-preserving system
$(\MMM,\BBB,\nu^*,\mbox{$\bfv$-flow})$ which is ergodic.
To prove this, we use the almost trivial fact that $\frakM_1$ is convex.
The general Krein--Millman theorem in functional analysis implies that the non-empty compact and convex set $\frakM_1$
is spanned by its extremal points.
It is a well known result in ergodic theory that the extremal points in the space of measures are the ergodic measures;
see Furstenberg \cite[Proposition~3.4]{furstenberg81}.
Thus we can choose the measure $\nu^*\in\frakM_1$ above to be such an extremal point.
\end{part1}

\begin{part2}
Consider the measure-preserving system $(\MMM,\BBB,\nu^*,\mbox{$\bfv$-flow})$ which is ergodic.
We now apply the flow version of the Birkhoff ergodic theorem.
For a continuous function $f$ on $\MMM$ and $\nu^*$-almost every starting point $\bfp\in\MMM$, we have
\begin{equation}\label{eq4.7}
\lim_{n\to\infty}\frac{1}{n}\int_0^nf(\LLL(\bfv;\bfp;t))\,\dd t=\int_\MMM f\,\dd\nu^*.
\end{equation}

We return to the continuous function $f_0$ at the beginning of Part 1 for which uniformity fails, and consider the Borel set
\begin{equation}\label{eq4.8}
Y=\left\{\bfp\in\MMM:\lim_{n\to\infty}\frac{1}{n}\int_0^nf_0(\LLL(\bfv;\bfp;t))\,\dd t=\int_\MMM f_0\,\dd\nu^*\right\}.
\end{equation}
Then it follows from \eqref{eq4.7} and \eqref{eq4.8} that
\begin{equation}\label{eq4.9}
\nu^*(Y)=1.
\end{equation}
Meanwhile, the weaker form of Theorem~\ref{thm1} gives uniformity for $\lambda_3$-almost every
starting point $\bfp\in\MMM$, and so
\begin{equation}\label{eq4.10}
\lim_{n\to\infty}\frac{1}{n}\int_0^nf_0(\LLL(\bfv;\bfp;t))\,\dd t=\frac{1}{3}\int_\MMM f_0\,\dd\lambda_3.
\end{equation}
Combining \eqref{eq4.6} with $\nu^*$ replacing $\nu_\infty$, \eqref{eq4.7} in the special case $f=f_0$ and \eqref{eq4.10},
we conclude that
\textcolor{white}{xxxxxxxxxxxxxxxxxxxxxxxxxxxxxx}
\begin{displaymath}
\lambda_3(Y)=0.
\end{displaymath}
This implies that for every $\delta>0$, there exists an infinite collection of disjoint axis-parallel rectangular boxes
$\frakR_i$, $i=1,2,3,\ldots,$ such that
\begin{equation}\label{eq4.11}
\sum_{i=1}^\infty\lambda_3(\frakR_i)<\delta
\quad\mbox{and}\quad
Y\subset\bigcup_{i=1}^\infty\frakR_i,
\end{equation}
and it follows from \eqref{eq4.9} that
\textcolor{white}{xxxxxxxxxxxxxxxxxxxxxxxxxxxxxx}
\begin{equation}\label{eq4.12}
\nu^*\left(\bigcup_{i=1}^\infty\frakR_i\right)=1.
\end{equation}
Comparing \eqref{eq4.11} and \eqref{eq4.12}, we see that there exists a positive integer $i_0$ such that
\begin{equation}\label{eq4.13}
\frac{\lambda_3(\frakR_{i_0})}{\nu^*(\frakR_{i_0})}<\delta.
\end{equation}
Clearly, for $\nu^*$-almost every starting point $\bfp\in\MMM$, we have
\begin{equation}\label{eq4.14}
\lim_{n\to\infty}\frac{1}{n}\int_0^n\chi_{\frakR_{i_0}}(\LLL(\bfv;\bfp;t))\,\dd t=\nu^*(\frakR_{i_0}),
\end{equation}
where $\chi_{\frakR_{i_0}}$ is the characteristic function of~$\frakR_{i_0}$.

We next make use of the crucial fact that the $\bfv$-flow in $\MMM$ reduces modulo one to $\bfv$-flow in the unit torus $[0,1)^3$.
Since $\MMM$ consists of $3$ cubes, every point in the unit torus $[0,1)^3$ has $3$ pre-images in~$\MMM$.
We can clearly assume that $\frakR_{i_0}$ is contained in one of these $3$ cubes.
Write $\frakR_{i_0}=\frakR(1)$, and then let $\frakR(2)$ and $\frakR(3)$ denote axis-parallel rectangular boxes
in the other $2$ cubes with the property that $\frakR(1),\frakR(2),\frakR(3)$ reduce modulo one to the same
axis-parallel rectangular box in the unit torus $[0,1)^3$.

Since $\bfv$ is a Kronecker direction, the continuous form of the Kronecker--Weyl equidistribution theorem applies for
geodesic flow in direction $\bfv$ in the unit torus $[0,1)^3$.
It follows that
\begin{equation}\label{eq4.15}
\lim_{n\to\infty}\frac{1}{n}\sum_{i=1}^3\int_0^n\chi_{\frakR(i)}(\LLL(\bfv;\bfp;t))\,\dd t
=\lambda_3(\frakR(1)\cup\frakR(2)\cup\frakR(3))
=3\lambda_3(\frakR_{i_0})
\end{equation}
holds for every starting point $\bfp\in\MMM$ where the half-infinite geodesic does not hit a singularity.

Recall that \eqref{eq4.14} holds for $\nu^*$-almost every starting point $\bfp\in\MMM$.
Let $\bfp^*$ be such a starting point.
Combining \eqref{eq4.13} and \eqref{eq4.14} gives
\begin{equation}\label{eq4.16}
\frac{\lambda_3(\frakR_{i_0})}{\delta}
<\lim_{n\to\infty}\frac{1}{n}\int_0^n\chi_{\frakR_{i_0}}(\LLL(\bfv;\bfp^*;t))\,\dd t.
\end{equation}
On the other hand, the right hand side of \eqref{eq4.16} can be bounded from above by using \eqref{eq4.15}
with $\bfp=\bfp^*$, and we have
\begin{displaymath}
\lim_{n\to\infty}\frac{1}{n}\int_0^n\chi_{\frakR_{i_0}}(\LLL(\bfv;\bfp^*;t))\,\dd t\le3\lambda_3(\frakR_{i_0}).
\end{displaymath}
Choosing $\delta=1/3$ clearly leads to a contradiction, and this completes the proof of Theorem~\ref{thm1}.
\end{part2}

%
%

\section{Proof of Theorem~\ref{thm2}}\label{sec5}

The proof of Theorem~\ref{thm2} is an adaptation of the proof of Theorem~\ref{thm1}, where the only extra idea is one more
application of ergodicity.

The geodesic flow in $\PPP\times[0,1)$ under consideration is $\bfv$-flow, with $\bfv=(\alpha,1,\beta)$, where
$\alpha$, $\beta$ and $\alpha/\beta$ are all irrational.
For simplicity, we assume that $\alpha>0$ and $\beta>0$, and like before, we further assume that $\alpha<1$.

For the $\sigma$-th atomic cube of $\PPP\times[0,1)$, where $\sigma=1,\ldots,d$, we denote by $X_\sigma$ and $Y_\sigma$
the left $X$-face and the bottom $Y$-face respectively,
and denote by $E_\sigma$ the common edge of the right $X$-face and the top $Y$-face, as shown in Figure~5.1.

\begin{displaymath}
\begin{array}{c}
\includegraphics[scale=0.8]{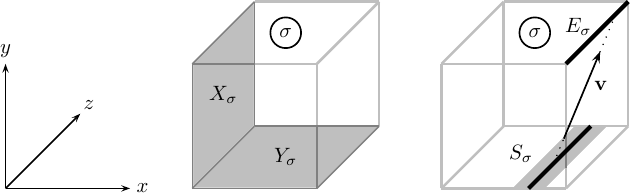}
\vspace{3pt}\\
\mbox{Figure 5.1: the faces $X_\sigma,Y_\sigma$ and the edge $E_\sigma$ of the $\sigma$-th atomic cube,}
\\
\mbox{and the narrow strip $S_\sigma$ on $Y_\sigma$}
\end{array}
\end{displaymath}

We consider the discretization of the $\bfv$-flow in polycube translation $3$-manifold $\PPP\times[0,1)$ relative to the
$Y$-faces, where the flow hits the faces $Y_1,\ldots,Y_d$.
Let
\begin{equation}\label{eq5.1}
T_{\alpha,\beta}:\YYY=Y_1\cup\ldots\cup Y_d\to\YYY
\end{equation}
denote the relevant discrete transformation defined by consecutive hitting points.
Note that the projection of the $\bfv$-flow to the $xy$-plane gives rise to a flow on the polysquare translation surface $\PPP$
with slope~$\alpha^{-1}$.

We need to prove that $T_{\alpha,\beta}$ is ergodic.
As before, we proceed by contradiction.
Assume, on the contrary, that $T_{\alpha,\beta}$ is not ergodic.
Then there exist two $T_{\alpha,\beta}$-invariant subsets $\WWW$ (white) and $\SSS$ (silver)
such that $\YYY=\WWW\cup\SSS$ and
\begin{equation}\label{eq5.2}
1\le\lambda_2(\WWW)\le\lambda_2(\SSS)\le d-1,
\end{equation}
where $\lambda_2$ denotes the $2$-dimensional Lebesgue measure.
We show that this leads to a contradiction.

Again, we consider the multiplicity functions $f_\WWW$ and $f_\SSS$ of the invariant sets $\WWW$ and $\SSS$ respectively.
More precisely, consider the projection of $\YYY$ to the unit torus $[0,1)^2$, given by \eqref{eq1.3}.
Then for every point $P\in[0,1)^2$, there are precisely $d$ points $P_1,\ldots,P_d\in\YYY$ which have projection image~$P$.
Let
\begin{equation}\label{eq5.3}
f_\WWW(P)=\vert\{P_1,\ldots,P_d\}\cap\WWW\vert
\quad\mbox{and}\quad
f_\SSS(P)=\vert\{P_1,\ldots,P_d\}\cap\SSS\vert
\end{equation}
denote the number of these $d$ points that fall into $\WWW$ and $\SSS$ respectively, so that
\begin{equation}\label{eq5.4}
f_\WWW(P)+f_\SSS(P)=d
\quad
\mbox{for almost all $P\in[0,1)^2$}.
\end{equation}
We also consider the corresponding projection of
\begin{displaymath}
T_{\alpha,\beta}:\YYY\to\YYY
\quad\mbox{to}\quad
T_0:[0,1)^2\to[0,1)^2,
\end{displaymath}
which is simply the $(\alpha,\beta)$-shift on the unit torus $[0,1)^2$.

For each $\sigma=1,\ldots,d$, there is a line segment on the face $Y_\sigma$ which is the image under $T_{\alpha,\beta}^{-1}$
of the edge $E_\sigma$ of $\PPP\times[0,1)$ of the $\sigma$-th atomic cube.
It is possible that this edge $E_\sigma$ is a singular edge of $\PPP\times[0,1)$.
Consider a narrow strip neighbourhood $S_\sigma$ of this line segment, as illustrated in light gray in the picture on the right in Figure~5.1.
As in the proof of Lemma~\ref{lem12}, we can think of the $x$-coordinate of every point in $S_\sigma$
to fall modulo one into an interval of the type
\begin{displaymath}
J_k(\alpha;-1)=\left(\{-\alpha\}-\frac{\Vert q_k\alpha\Vert}{2},\{-\alpha\}+\frac{\Vert q_k\alpha\Vert}{2}\right).
\end{displaymath}
In other words, if $\bfp_\sigma$ represents the bottom left vertex of the face $Y_\sigma$ of the translation $3$-manifold~$\MMM$, then
\begin{displaymath}
S_\sigma=\{\bfp_\sigma+(x,0,z)\in Y_\sigma:(x,z)\in J_k(\alpha;-1)\times[0,1)\}.
\end{displaymath}

It is straightforward to extend Lemma~\ref{lem12} to the case of a polycube translation $3$-manifold
which is the cartesian product of a polysquare translation surface with $d$ faces with the unit torus.
However, to establish Theorem~\ref{thm2}, we need something stronger than this extension,
and this is facilitated by making a small change to the definition of the small special rectangles.

For every integer $j=1,\ldots,q'_{h+1}-1$, we replace the short special interval \eqref{eq3.8} by the possibly longer interval
\begin{equation}\label{eq5.5}
J'_h(j)=J'_h(\beta;j)
=\left(\{j\beta\}-\frac{\frakd_h(\beta;j;-)}{2},\{j\beta\}+\frac{\frakd_h(\beta;j;+)}{2}\right),
\end{equation}
where $\frakd(\beta;j;-)$ and $\frakd(\beta;j;+)$ denote the distances of $\{j\beta\}$ to its immediate left neighbour
and to its immediate right neighbour respectively in $\BBB_h(\beta)$.
Clearly each of $\frakd(\beta;j;-)$ and $\frakd(\beta;j;+)$ has value equal to one of the values in \eqref{eq3.7}.

The small special rectangles $R_{\sigma,k,h}$ and $R^\dagger_{\sigma,k,h}$ as well as the $T_{\alpha,\beta}$-power chains
and the $T^{-1}_{\alpha,\beta}$-power chains are then defined with the modification \eqref{eq5.5} in place.
Note in particular that this modification is in the $z$-direction which is \textit{integrable},
so all the small special rectangles and their images remain splitting free as before.

Let $I\subset[0,1)$ be an interval of the form \eqref{eq5.5}, and let
\begin{equation}\label{eq5.6}
R_\sigma=\{(x,y,z)\in S_\sigma:z\in I\},
\quad
\sigma=1,\ldots,d,
\end{equation}
denote small special rectangles of the decomposition of the strips $S_1,\ldots,S_d$
on the faces $Y_1,\ldots,Y_d$ respectively with the condition that their images on $[0,1)^2$ under the projection
\eqref{eq1.3} coincide, so that they are in the same relative positions on the faces $Y_1,\ldots,Y_d$ respectively.
It is clear that the image of each of $R_1,\ldots,R_d$ under $T_{\alpha,\beta}$ splits along the edges $E_1,\ldots,E_d$
into congruent halves.
For each $\sigma=1,\ldots,d$, the rectangle $R_\sigma$ splits into the left half $R_\sigma^-$ and the right half~$R_\sigma^+$,
and it is clear that their images under $T_{\alpha,\beta}$ may lie on distinct $Y$-faces.
On the other hand, for every $\sigma=1,\ldots,d$, there is a small special rectangle in $Y_\sigma$ of the form
\begin{displaymath}
R^*_\sigma=T_{\alpha,\beta}^2(R_{\rho'}^-)\cup T_{\alpha,\beta}^2(R_{\rho''}^+),
\end{displaymath}
which is the union of the image under $T_{\alpha,\beta}^2$ of the left half $R_{\rho'}^-$ of some small special rectangle $R_{\rho'}$
and the image under $T_{\alpha,\beta}^2$ of the right half $R_{\rho''}^+$ of some small special rectangle~$R_{\rho''}$.
Here the two small special rectangles $R_{\rho'}$ and $R_{\rho''}$ may be identical or different.
However, in view of the restriction imposed on the small special rectangles $R_1,\ldots,R_d$ that their images under
the projection \eqref{eq1.3} coincide, we deduce that the images of the small special rectangles $R^*_1,\ldots,R^*_d$
under the projection \eqref{eq1.3} also coincide, apart possibly from sets of zero measure caused by
the singular edges of $\PPP\times[0,1)$.

\begin{displaymath}
\begin{array}{c}
\includegraphics[scale=0.8]{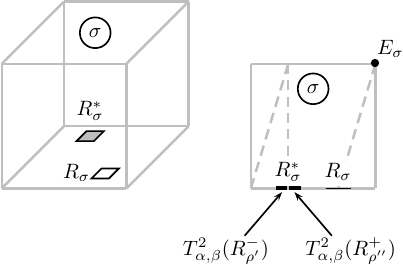}
\vspace{3pt}\\
\mbox{Figure 5.2: the small special rectangles $R_\sigma$ and $R^*_\sigma$ on the face $Y_\sigma$,}
\\
\mbox{and the projectied view with the $z$-coordinate removed}
\end{array}
\end{displaymath}

It is clear that the rectangles $R_1,\ldots,R_d$ and $R^*_1,\ldots,R^*_d$ have the same area.

Given a positive number $\eta>0$, we say that a collection $(R_1,\ldots,R_d)$ of the form \eqref{eq5.6} is $\eta$-nice if
for each $\sigma=1,\ldots,d$, the set $R_\sigma$ satisfies
\begin{displaymath}
\frac{\lambda_2(R_\sigma\cap\WWW)}{\lambda_2(R_\sigma)}>1-\eta
\quad\mbox{or}\quad
\frac{\lambda_2(R_\sigma\cap\SSS)}{\lambda_2(R_\sigma)}>1-\eta,
\end{displaymath}
and the set $R^*_\sigma$ satisfies
\begin{displaymath}
\frac{\lambda_2(R^*_\sigma\cap\WWW)}{\lambda_2(R^*_\sigma)}>1-\eta
\quad\mbox{or}\quad
\frac{\lambda_2(R^*_\sigma\cap\SSS)}{\lambda_2(R^*_\sigma)}>1-\eta.
\end{displaymath}

We have the following stronger form of Lemma~\ref{lem12}.

\begin{lemma}\label{lem51}
Let $\eps>0$ and $\eta>0$ be arbitrarily small and fixed.
Then there exist sufficiently large integers $k$ and $h$ such that for every $\sigma=1,\ldots,d$,
the union of all rectangles $R_{\sigma,k,h}(\alpha;1;\beta;j)$ in the strip $S_\sigma$ arising from $\eta$-nice collections
$(R_1,\ldots,R_d)$ has area at least $(1-\eps)\lambda_2(S_\sigma)$.
\end{lemma}

\begin{proof}[Idea of proof]
There are two critical changes to the proof of Lemma~\ref{lem12}, and each contributes to part of the value of~$\eps$.

First of all, by modifying the small rectangles in the $z$-direction,
we ensure that for each $\sigma=1,\ldots,d$, the inequality
\begin{displaymath}
\lambda_2\left(\bigcup_{j=1}^{q'_{h+1}-q_{k+1}}R_{\sigma,k,h}(\alpha;1;\beta;j)\right)
\ge\left(1-q_{k+1}(\Vert q'_h\beta\Vert+\Vert q'_{h+1}\beta\Vert)\right)\lambda_2(S_\sigma)
\end{displaymath}
holds.
With sufficiently large values of $h$ and~$k$, the factor $q_{k+1}(\Vert q'_h\beta\Vert+\Vert q'_{h+1}\beta\Vert)$
can be made arbitrarily close to~$0$.
This means that the small special rectangles in the strip $S_\sigma$ make up nearly all of~$S_\sigma$.

Next, the conclusion of Lemma~\ref{lem32} can be significantly extended.
It can be shown that by choosing $h$ and $k$ suitably large, we can guarantee a proportion arbitrarily close to $1$ of indices $j_0$
satisfying $q_{k+1}+1\le j_0\le q'_{h+1}-q_{k+1}$ such that for every $\sigma=1,\ldots,d$,
the $T_{\alpha,\beta}$-power chain \eqref{eq3.13} with $j=j_0$ is not defective, and
the $T_{\alpha,\beta}^{-1}$-power chain \eqref{eq3.22} with $j=j_0$ is not defective.
\end{proof}

We may assume, without loss of generality, that $\sigma=1$ is part of the cycle of length $d_0>d/2$ of the
splitting permutation.

Suppose that there exist two different small special rectangles $R'_1$ and $R''_1$ in the strip $S_1$ that arise from
$\eta$-nice collections $(R'_1,\ldots,R'_d)$ and $(R''_1,\ldots,R''_d)$ respectively such that
\textcolor{white}{xxxxxxxxxxxxxxxxxxxxxxxxxxxxxx}
\begin{displaymath}
\frac{\lambda_2(R'_1\cap\WWW)}{\lambda_2(R'_1)}>1-\eta
\quad\mbox{and}\quad
\frac{\lambda_2(R''_1\cap\SSS)}{\lambda_2(R''_1)}>1-\eta.
\end{displaymath}
Applying the splitting method argument at the end of Section~\ref{sec1} to~$R'_1$,
we conclude that there are at least $d_0$ sets among $R'_1,\ldots,R'_d$ which are overwhelmingly in~$\WWW$,
so that $f_\WWW\ge d_0>d/2$.
Meanwhile, applying the same argument to~$R''_1$,
we conclude that there are at least $d_0$ sets among $R''_1,\ldots,R''_d$ which are overwhelmingly in~$\SSS$,
so that $f_\SSS\ge d_0>d/2$.
This is clearly a contradiction, since $f_\WWW+f_\SSS=d$.

Hence the strip $S_1$ is overwhelmingly monochromatic.
We may assume, without loss of generality, that the strip $S_1$ is overwhelmingly in~$\WWW$.
In fact, a proportion exceedingly $(1-\eps)(1-\eta)$ of $S_1$ is in~$\WWW$.
We shall show that this eventually leads to a contradiction.

We return to the translation $3$-manifold $\MMM=\PPP\times[0,1)$, with $\lambda_3$-preserving flow in the direction
$\bfv=(\alpha,1,\beta)$.
We have assumed that this flow is not ergodic, so that there exists a decomposition $\MMM=\WWW_3\cup\SSS_3$
into two $\bfv$-flow-invariant and disjoint subsets $\WWW_3$ and $\SSS_3$ such that
\begin{displaymath}
1\le\lambda_3(\WWW_3)=f_\WWW\le d-1,
\quad
1\le\lambda_3(\SSS_3)=f_\SSS\le d-1,
\quad
f_\WWW+f_\SSS=d.
\end{displaymath}

\begin{remark}
In view of the relationship between the discrete $T_{\alpha,\beta}$-invariance and the continuous $\bfv$-flow-invariance,
one can visualize a close relationship between the pair $(\WWW,\SSS)$ and the pair $(\WWW_3,\SSS_3)$, in that
\begin{displaymath}
f_\WWW(P)=\lambda_3(\WWW_3)
\quad\mbox{and}\quad
f_\SSS(P)=\lambda_3(\SSS_3)
\quad\mbox{for almost all $P\in[0,1)^2$}.
\end{displaymath}
For convenience, we may write $f_\WWW=\lambda_3(\WWW_3)$ and $f_\SSS=\lambda_3(\SSS_3)$.
\end{remark}

Let $\chi=\chi_{\WWW_3}$ denote the characteristic function of the invariant subset $\WWW_3\subset\MMM$,
so that
\textcolor{white}{xxxxxxxxxxxxxxxxxxxxxxxxxxxxxx}
\begin{displaymath}
\chi(\bfs)=\left\{\begin{array}{ll}
1,&\mbox{if $\bfs\in\WWW_3$},\\
0,&\mbox{if $\bfs\not\in\WWW_3$}.
\end{array}\right.
\end{displaymath}
Writing $\bfs=(x,y,z)$ and using Fubini's theorem, we have
\begin{displaymath}
\int_\MMM\chi(\bfs)\,\dd\bfs
=\int_\PPP\left(\int_{[0,1)}\chi(x,y,z)\,\dd z\right)\dd\lambda_2(x,y),
\end{displaymath}
where the inner integral
\textcolor{white}{xxxxxxxxxxxxxxxxxxxxxxxxxxxxxx}
\begin{equation}\label{eq5.7}
\psi(x,y)=\int_{[0,1)}\chi(x,y,z)\,\dd z
\end{equation}
is well defined in the Lebesgue sense for almost every $(x,y)\in\PPP$.

\begin{lemma}\label{lem52}
Suppose that the $\bfv$-flow in $\MMM$ is not ergodic.
Then for almost every $(x,y)\in\PPP$, we have
\textcolor{white}{xxxxxxxxxxxxxxxxxxxxxxxxxxxxxx}
\begin{equation}\label{eq5.8}
\frac{1}{d}\le\psi(x,y)=\frac{f_\WWW}{d}\le\frac{d-1}{d}.
\end{equation}
\end{lemma}

\begin{proof}[Sketch of proof]
Recall that the projection of the $\bfv$-flow in the translation $3$-manifold $\MMM$ to the $xy$-plane is a geodesic flow with slope $\alpha^{-1}$
on the polysquare translation surface~$\PPP$.
Meanwhile, note that $\MMM=\PPP\times[0,1)$ is a cartesian product where the second factor, the unit torus $[0,1)$,
is in the \textit{integrable} $z$-direction.
Since $\WWW_3\subset\MMM$ is a $\bfv$-flow-invariant set, it follows that the function $\psi(x,y)$ in \eqref{eq5.7} on $\PPP$
is invariant under the geodesic flow with slope $\alpha^{-1}$ on~$\PPP$.
As the slope $\alpha^{-1}$ is irrational, it follows from the Gutkin--Veech theorem \cite{gutkin84,veech87}
that this geodesic flow on $\PPP$ is ergodic.
This implies that the function $\psi(x,y)$ is constant almost everywhere on~$\PPP$.
The result follows.
\end{proof}

Lemma~\ref{lem52} implies that $f_\SSS/d\ge1/d$ part of $S_1$ is in~$\SSS$, a contradiction.
This gives ergodicity in Theorem~\ref{thm2}.
Unique ergodicity can be established by a straightforward adaptation of the argument in Section~\ref{sec4}.

We also establish the following simple consequence for use later.

Let $\partial_X\PPP$ denote the union of the $d$ edges of the polysquare surface $\PPP$ which are in the $x$-direction.

\begin{lemma}\label{lem53}
Suppose that the $\bfv$-flow in $\MMM$ is not ergodic.
Then for almost every point $(x,y)\in\partial_X\PPP$, the inequalities \eqref{eq5.8} hold.
\end{lemma}

\begin{proof}[Sketch of proof]
Since the $\bfv$-flow is linear, the integral \eqref{eq5.7} is well defined for almost every $(x,y)\in\partial_X\PPP$.
On the other hand, for every $x^*\in[0,1)$, there are precisely $d$ points $(x_\sigma,y_\sigma)\in\partial_X\PPP$,
$\sigma=1,\ldots,d$, such that $x_\sigma\equiv x^*\bmod{1}$.
\end{proof}

%
%

\section{A stronger result}\label{sec6}

The next step in our investigation is to remove the somewhat artificial hypothesis in Theorem~\ref{thm2}
concerning the splitting permutation.

\begin{theorem}\label{thm3}
Suppose that a polycube translation $3$-manifold $\MMM$ is the cartesian product of a polysquare translation surface $\PPP$
with $d$ square faces and the unit torus $[0,1)$.
Then any half-infinite geodesic in $\MMM$ with a Kronecker direction $\bfv\in\Rr^3$ is uniformly distributed unless it hits a singularity.
\end{theorem}

As before for Theorem~\ref{thm2}, we assume that $\bfv=(\alpha,1,\beta)$, where $\alpha$, $\beta$ and $\alpha/\beta$ are all irrational,
and that $0<\alpha<1$ and $\beta>0$.
We also adopt the notation and terminology in Section~\ref{sec5}.
Again, we assume that the discrete transformation $T_{\alpha,\beta}$ given by \eqref{eq5.1} is not ergodic.
This leads to a partition $\YYY=\WWW\cup\SSS$ into a disjoint union
of two $T_{\alpha,\beta}$-invariant measurable subsets $\WWW$ and $\SSS$ such that \eqref{eq5.2}--\eqref{eq5.4} are valid.
Since $\WWW$ is measurable, given any arbitrarily small $\eps>0$, to be fixed later, there exists a set
$\WWW_1=\WWW_1(\eps;\WWW)$ which is a finite union of disjoint rectangles aligned with the edges of the faces
$Y_1,\ldots,Y_d$ such that the symmetric difference
\begin{displaymath}
\WWW\bigtriangleup\WWW_1=(\WWW\setminus\WWW_1)\cup(\WWW_1\setminus\WWW)
\end{displaymath}
satisfies the inequality
\textcolor{white}{xxxxxxxxxxxxxxxxxxxxxxxxxxxxxx}
\begin{equation}\label{eq6.1}
\lambda_2(\WWW\bigtriangleup\WWW_1)<\eps.
\end{equation}

Given any positive integer~$H$, we can divide the unit interval $[0,1)$ in the usual way into $H$ equal subintervals
\begin{displaymath}
I(\tau)=\left[\frac{\tau-1}{H},\frac{\tau}{H}\right),
\quad
\tau=1,\ldots,H.
\end{displaymath}
For each $\tau=1,\ldots,H$ and $\sigma=1,\ldots,d$, let $I_\sigma(\tau)$ be the subinterval on the part of $\partial_X\PPP$ corresponding to
the face $Y_\sigma$ such that $I_\sigma(\tau)\equiv I(\tau)\bmod{1}$, and consider the strip $I_\sigma(\tau)\times\{y_\sigma\}\times[0,1)$
on the face~$Y_\sigma$, where $y_\sigma$ is the common value of the $y$-coordinates of the points on~$Y_\sigma$.

Since the set $\WWW_1$ is a finite union of disjoint aligned rectangles, it follows that as long as the threshold $H=H(\WWW_1)$
is chosen sufficiently large, then at least half of the indices $\tau=1,\ldots,H$ are \textit{safe} in the sense that
for every $\sigma=1,\ldots,d$, the strip $I_\sigma(\tau)\times\{y_\sigma\}\times[0,1)$ on the face~$Y_\sigma$ does not contain any $z$-parallel edge
of any aligned rectangle in the finite union~$\WWW_1$.
Let $\HHH\subset\{1,\ldots,H\}$ denote the collection of safe indices.
Then it is clear that $\vert\HHH\vert\ge H/2$.

Let $\tau^*\in\HHH$ satisfy the property
\begin{align}
&
\lambda_2\left((\WWW\bigtriangleup\WWW_1)\cap\bigcup_{\sigma=1}^d(I_\sigma(\tau^*)\times\{y_\sigma\}\times[0,1))\right)
\nonumber
\\
&\quad
=\min_{\tau\in\HHH}\lambda_2\left((\WWW\bigtriangleup\WWW_1)\cap\bigcup_{\sigma=1}^d(I_\sigma(\tau)\times\{y_\sigma\}\times[0,1))\right).
\nonumber
\end{align}
For convenience, write
\begin{equation}\label{eq6.2}
I(\tau^*)=\III_0=\III_0(\WWW_1)\subset[0,1),
\quad\mbox{so that}\quad
\vert\III_0\vert=H^{-1},
\end{equation}
and for $\sigma=1,\ldots,d$, write
\begin{displaymath}
\III_\sigma=I_\sigma(\tau^*),
\quad\mbox{so that}\quad
\III_\sigma\equiv\III_0\bmod{1}.
\end{displaymath}
We have the following intermediate results.

\begin{pa}
For every $\sigma=1,\ldots,d$ and every $x\in\III_\sigma$, consider the $z$-parallel cross section of $\WWW_1$
on the face $Y_\sigma$ given by
\begin{displaymath}
\cs_\sigma(x;\WWW_1)=\{z\in[0,1):(x,y_\sigma,z)\in\WWW_1\cap Y_\sigma\}\subset[0,1).
\end{displaymath}
Then $\cs_\sigma(x;\WWW_1)$ is independent of the choice of $x\in\III_\sigma$.
\end{pa}

\begin{proof}[Sketch of proof]
Consider the set
\begin{displaymath}
\cs_\sigma(\WWW_1)=\{(x,z)\in\III_\sigma\times[0,1):(x,y_\sigma,z)\in\WWW_1\cap Y_\sigma\}.
\end{displaymath}
Note that $\III_\sigma=I_\sigma(\tau^*)$, where $\tau^*$ is a safe index.
This means that for any aligned rectangle $\frakR$ in the finite union~$\WWW_1$, either
\begin{displaymath}
\cs_\sigma(\WWW_1)\cap\frakR=\emptyset
\quad\mbox{or}\quad
\cs_\sigma(\WWW_1)\cap\frakR=\III_\sigma\times\frakJ
\end{displaymath}
for some subinterval $\frakJ\subset[0,1)$.
The result follows almost immediately.
\end{proof}

\begin{remark}
For every $\sigma=1,\ldots,d$ and almost every $x\in\III_\sigma$, the $z$-parallel cross section of $\WWW$
on the face $Y_\sigma$ given by
\begin{displaymath}
\cs_\sigma(x;\WWW)=\{z\in[0,1):(x,y_\sigma,z)\in\WWW\cap Y_\sigma\}\subset[0,1)
\end{displaymath}
is well defined and measurable, but is not independent of the choice of $x\in\III_\sigma$.
\end{remark}

\begin{pb}
For every $\sigma=1,\ldots,d$, we have
\begin{equation}\label{eq6.3}
\lambda_2((\WWW\bigtriangleup\WWW_1)\cap(\III_\sigma\times\{y_\sigma\}\times[0,1)))<\frac{2\eps}{H}.
\end{equation}
\end{pb}

\begin{proof}
We have
\begin{align}
&
\lambda_2((\WWW\bigtriangleup\WWW_1)\cap(\III_\sigma\times\{y_\sigma\}\times[0,1)))
=\lambda_2((\WWW\bigtriangleup\WWW_1)\cap(I_\sigma(\tau^*)\times\{y_\sigma\}\times[0,1)))
\nonumber
\\
&\quad
\le\lambda_2\left((\WWW\bigtriangleup\WWW_1)\cap\bigcup_{\sigma=1}^d(I_\sigma(\tau^*)\times\{y_\sigma\}\times[0,1))\right)
\le\frac{\lambda_2(\WWW\bigtriangleup\WWW_1)}{\vert\HHH\vert}
<\frac{2\eps}{H},
\nonumber
\end{align}
where the second last inequality comes from the fact that the minimum over $\HHH$ does not exceed the average
and the fact that
\begin{displaymath}
\bigcup_{\tau\in\HHH}\bigcup_{\sigma=1}^d(I_\sigma(\tau)\times\{y_\sigma\}\times[0,1))\subset\YYY
\end{displaymath}
is a disjoint union so that
\begin{align}
&
\sum_{\tau\in\HHH}\lambda_2\left((\WWW\bigtriangleup\WWW_1)\cap\bigcup_{\sigma=1}^d(I_\sigma(\tau)\times\{y_\sigma\}\times[0,1))\right)
\nonumber
\\
&\quad
\le\lambda_2((\WWW\bigtriangleup\WWW_1)\cap\YYY)
=\lambda_2(\WWW\bigtriangleup\WWW_1).
\nonumber
\end{align}
This completes the proof.
\end{proof}

In view of Property~A, for every $\sigma=1,\ldots,d$, we can choose $x_\sigma\in\III_\sigma$ arbitrarily, and consider
the $z$-parellel cross section
\begin{displaymath}
U_\sigma=\cs_\sigma(x_\sigma;\WWW_1).
\end{displaymath}

We need some auxiliary results.

We consider Lebesgue measurable subsets $U_\sigma\subset[0,1)$, $\sigma=1,\ldots,d$, of the unit torus $[0,1)$.
In particular, we make the assumption that $0<\lambda_1(U_1)<1$, where $\lambda_1$ denotes $1$-dimensional Lebesgue measure.
Furthermore, for any real number $u\in\Rr$, we consider the $u$-translated copy of~$U_1$, given by
\begin{displaymath}
u+U_1=\{\{u+z\}:z\in U_1\}.
\end{displaymath}

\begin{lemma}\label{lem61}
The set of values $u_0\in[0,1)$ for which the inequalities
\begin{displaymath}
\lambda_1(U_\sigma\bigtriangleup(u_0+U_1))\ge\frac{1}{32d^2}\lambda_1(U_1)(1-\lambda_1(U_1)),
\quad
\sigma=1,\ldots,d,
\end{displaymath}
hold simultaneously has Lebesgue measure at least~$1/2$.
\end{lemma}

For every real number $z\in\Rr$, denote by
\begin{displaymath}
\Vert z\Vert=\min_{n\in\Zz}\vert z-n\vert
\end{displaymath}
the distance of $z$ from the nearest integer.

\begin{lemma}\label{lem62}
Let $\bfv=(\alpha,1,\beta)\in\Rr^3$ be a Kronecker vector.
For every $\eps_1>0$, there exists an infinite sequence
\begin{displaymath}
1\le m_1(\eps_1)<m_2(\eps_1)<m_3(\eps_1)<\ldots
\end{displaymath}
of positive integers such that
\begin{equation}\label{eq6.4}
\Vert m_j(\eps_1)\alpha\Vert<\eps_1,
\quad
j=1,2,3,\ldots,
\end{equation}
and the sequence $\{m_j(\eps_1)\beta\}$, $j=1,2,3,\ldots,$ is uniformly distributed in $[0,1)$.
\end{lemma}

\begin{proof}
Since $\alpha$, $\beta$ and $\alpha/\beta$ are all irrational, the Kronecker--Weyl equidistribution theorem implies that
the sequence
\begin{displaymath}
(\{j\alpha\},\{j\beta\}),
\quad
j=1,2,3,\ldots,
\end{displaymath}
is uniformly distributed in the unit square $[0,1)^2$.
For any $\eps_1>0$, let $m_1,m_2,m_3,\ldots$ be the subsequence of $1,2,3,\ldots$ such that
\begin{displaymath}
\{m_j\alpha\}\in[0,\eps_1)\cup(1-\eps_1,1),
\quad
j=1,2,3,\ldots.
\end{displaymath}
The conclusion follows easily.
\end{proof}

Since the sequence $\{m_j(\eps_1)\beta\}$, $j=1,2,3,\ldots,$ is uniformly distributed in $[0,1)$,
it follows that for any $\delta>0$, there exists a subsequence $m_{j_1},m_{j_2},m_{j_3},\ldots$ such that for each term
$m_{j_t}=m_{j_t}(\eps_1;\delta)$, the inequality
\begin{displaymath}
\Vert m_{j_t}(\eps_1;\delta)\beta-u_0\Vert<\delta
\end{displaymath}
holds, where $u_0\in[0,1)$ is given by Lemma~\ref{lem61} and fixed.
Clearly there exist two successive convergents of~$\alpha$, with denominators $q_k(\alpha)$ and $q_{k+1}(\alpha)$ such that
\begin{equation}\label{eq6.5}
q_k(\alpha)\le m_{j_t}(\eps_1;\delta)<q_{k+1}(\alpha).
\end{equation}
We work with the denominator $q_k=q_k(\alpha)$ of a particular convergent of $\alpha$ that satisfies the inequalities \eqref{eq6.5},
and can choose $m_{j_t}(\eps_1;\delta)$ large enough so that $q_k$ is substantially larger than $\vert\III_0\vert^{-1}=H$,
where $\III_0$ is defined by \eqref{eq6.2}.

Write $N=m_{j_t}(\eps_1;\delta)$, so that
\begin{displaymath}
q_k(\alpha)\le N<q_{k+1}(\alpha).
\end{displaymath}

For every $\sigma=1,\ldots,d$ and $i=1,\ldots,q_{k+1}-1$, consider the $J_k$-type $z$-parallel strip
\begin{displaymath}
S_{\sigma,k}(i)=\{(x,y_\sigma,z)\in Y_\sigma:x\in J_k(i)\bmod{1}\},
\end{displaymath}
where the short special interval $J_k(i)$ is given by \eqref{eq3.4}.
Furthermore, consider the half $J_k$-type $z$-parallel strips
\begin{align}
&
S^-_{\sigma,k}(i)=\{(x,y_\sigma,z)\in Y_\sigma:x\in J^-_k(i)\bmod{1}\},
\nonumber
\\
&
S^+_{\sigma,k}(i)=\{(x,y_\sigma,z)\in Y_\sigma:x\in J^+_k(i)\bmod{1}\},
\nonumber
\end{align}
where
\begin{displaymath}
J^-_k(i)=\left(\{i\alpha\}-\frac{\Vert q_k\alpha\Vert}{2},\{i\alpha\}\right)
\quad\mbox{and}\quad
J^+_k(i)=\left(\{i\alpha\},\{i\alpha\}+\frac{\Vert q_k\alpha\Vert}{2}\right).
\end{displaymath}

Since $q_k$ is large compared to $\vert\III_0\vert^{-1}$, the $z$-parallel strip $\III_1\times\{y_1\}\times[0,1)\subset Y_1$
contains a significant number of strips of the form $S_{1,k}(i)$, where $i=1,\ldots,q_{k+1}-1$.

\begin{lemma}\label{lem63}
For every $i=1,\ldots,q_{k+1}-1$, at least one of the two $T_{\alpha,\beta}$-power chains
\begin{align}
&
T_{\alpha,\beta}^nS^-_{1,k}(i),
\quad
n=1,\ldots,N,
\label{eq6.6}
\\
&
T_{\alpha,\beta}^nS^+_{1,k}(i),
\quad
n=1,\ldots,N,
\label{eq6.7}
\end{align}
is splitting free.
\end{lemma}

\begin{proof}
For any fixed $i=1,\ldots,q_{k+1}-1$, consider the $T_{\alpha,\beta}$-power chain
\begin{displaymath}
T_{\alpha,\beta}^nS_{1,k}(i),
\quad
n=1,\ldots,N.
\end{displaymath}
Either no member of this chain splits, or there is a smallest $n_0=1,\ldots,N$ satisfying $i+n_0\ge q_{k+1}$,
in view of the $2$-distance theorem, such that $T_{\alpha,\beta}^{n_0}S_{1,k}(i)$ splits.
In the latter case, we have
\begin{equation}\label{eq6.8}
\left(\{(i+n_0)\alpha\}-\frac{\Vert q_k\alpha\Vert}{2},\{(i+n_0)\alpha\}+\frac{\Vert q_k\alpha\Vert}{2}\right)\cap\{0,1\}\ne\emptyset.
\end{equation}

Suppose first that the intersection \eqref{eq6.8} is~$\{0\}$.
Then
\begin{displaymath}
0<\{(i+n_0)\alpha\}<\frac{\Vert q_k\alpha\Vert}{2}<\{(i+n_0)\alpha\}+\frac{\Vert q_k\alpha\Vert}{2}<1.
\end{displaymath}
Any point $(x,y,z)\in T_{\alpha,\beta}^{n_0}S^+_{1,k}(i)$ must satisfy
\begin{equation}\label{eq6.9}
x\in\left(\{(i+n_0)\alpha\},\{(i+n_0)\alpha\}+\frac{\Vert q_k\alpha\Vert}{2}\right)\bmod{1}.
\end{equation}
Since the interval in \eqref{eq6.9} is contained in $(0,1)$, it follows that $T_{\alpha,\beta}^{n_0}S^+_{1,k}(i)$ does not split.
We can now prove by induction that for every $n=n_0+1,\ldots,N$, any point $(x,y,z)\in T_{\alpha,\beta}^nS^+_{1,k}(i)$ must satisfy
\begin{equation}\label{eq6.10}
x\in\left(\{(i+n)\alpha\},\{(i+n)\alpha\}+\frac{\Vert q_k\alpha\Vert}{2}\right)\bmod{1}.
\end{equation}
The inequality $0<n-n_0<q_{k+1}-1$ and the $2$-distance theorem now ensure that the interval in \eqref{eq6.10} is contained in $(0,1)$.
It follows that $T_{\alpha,\beta}^nS^+_{1,k}(i)$ does not split.
We thus conclude that the $T_{\alpha,\beta}$-power chain \eqref{eq6.7} does not split.

Suppose next that the intersection \eqref{eq6.8} is~$\{1\}$.
Then a similar argument shows that the $T_{\alpha,\beta}$-power chain \eqref{eq6.6} does not split.
\end{proof}

We denote by
\textcolor{white}{xxxxxxxxxxxxxxxxxxxxxxxxxxxxxx}
\begin{equation}\label{eq6.11}
T_{\alpha,\beta}^nS^\dagger_{1,k}(i),
\quad
n=1,\ldots,N,
\end{equation}
one of the $T_{\alpha,\beta}$-power chain \eqref{eq6.6} and \eqref{eq6.7} that is splitting free.

It is clear from the properties of the short special intervals $J_k(i)$, $i=1,\ldots,q_{k+1}-1$,
that the sets $S^\dagger_{1,k}(i)$, $i=1,\ldots,q_{k+1}-1$, that are contained in $\III_1\times\{y_1\}\times[0,1)\subset Y_1$
together occupy more than $1/10$ of $\III_1\times\{y_1\}\times[0,1)$, and so have total area exceeding $1/10H$.
For each of these sets, consider the last element of the chain \eqref{eq6.11}, of the form
\textcolor{white}{xxxxxxxxxxxxxxxxxxxxxxxxxxxxxx}
\begin{displaymath}
T_{\alpha,\beta}^NS^\dagger_{1,k}(i)\subset Y_\sigma
\quad\mbox{for some $\sigma=1,\ldots,d$}.
\end{displaymath}
Suppose that $Y_{\sigma_0}$ contains the maximum number of such images, and let $\frakI(\sigma_0)$
denote the collection of indices $i=1,\ldots,q_{k+1}-1$ such that $T_{\alpha,\beta}^NS^\dagger_{1,k}(i)\subset Y_{\sigma_0}$
and $S^\dagger_{1,k}(i)\subset\III_1\times\{y_1\}\times[0,1)$.
Then
\begin{equation}\label{eq6.12}
\bigcup_{i\in\frakI(\sigma_0)}T_{\alpha,\beta}^NS^\dagger_{1,k}(i)\subset Y_{\sigma_0}
\quad\mbox{and}\quad
\lambda_2\left(\bigcup_{i\in\frakI(\sigma_0)}T_{\alpha,\beta}^NS^\dagger_{1,k}(i)\right)\ge\frac{1}{10dH}.
\end{equation}

\begin{lemma}\label{lem64}
We have
\begin{equation}\label{eq6.13}
\lambda_2\left((\WWW\bigtriangleup\WWW_1)\cap\bigcup_{i\in\frakI(\sigma_0)}T_{\alpha,\beta}^NS^\dagger_{1,k}(i)\right)
<2\eps_1+20\eps d\lambda_2\left(\bigcup_{i\in\frakI(\sigma_0)}T_{\alpha,\beta}^NS^\dagger_{1,k}(i)\right),
\end{equation}
and
\textcolor{white}{xxxxxxxxxxxxxxxxxxxxxxxxxxxxxx}
\begin{equation}\label{eq6.14}
\lambda_2\left((\WWW\bigtriangleup\WWW_1)\cap\bigcup_{i\in\frakI(\sigma_0)}S^\dagger_{1,k}(i)\right)
<20\eps d\lambda_2\left(\bigcup_{i\in\frakI(\sigma_0)}S^\dagger_{1,k}(i)\right).
\end{equation}
\end{lemma}

\begin{proof}
For every $i\in\frakI(\sigma_0)$, the condition $S^\dagger_{1,k}(i)\subset\III_1\times\{y_1\}\times[0,1)$,
together with \eqref{eq6.4} and $N=m_{j_t}(\eps_1;\delta)$, ensure
that the total area of the images in \eqref{eq6.12} that fall outside the subset
$\III_{\sigma_0}\times\{y_{\sigma_0}\}\times[0,1)\subset Y_{\sigma_0}$ does not exceed~$2\eps_1$.
It then follows from \eqref{eq6.3} with $\sigma=\sigma_0$ and the second inequality in \eqref{eq6.12} that
\begin{align}
&
\lambda_2\left((\WWW\bigtriangleup\WWW_1)\cap\bigcup_{i\in\frakI(\sigma_0)}T_{\alpha,\beta}^NS^\dagger_{1,k}(i)\right)
\nonumber
\\
&\quad
\le\lambda_2\left(\left(\bigcup_{i\in\frakI(\sigma_0)}T_{\alpha,\beta}^NS^\dagger_{1,k}(i)\right)
\setminus(\III_{\sigma_0}\times\{y_{\sigma_0}\}\times[0,1))\right)
\nonumber
\\
&\quad\qquad
+\lambda_2((\WWW\bigtriangleup\WWW_1)\cap(\III_{\sigma_0}\times\{y_{\sigma_0}\}\times[0,1)))
\nonumber
\\
&\quad
<2\eps_1+\frac{2\eps}{H}
\le2\eps_1+20\eps d\lambda_2\left(\bigcup_{i\in\frakI(\sigma_0)}T_{\alpha,\beta}^NS^\dagger_{1,k}(i)\right),
\nonumber
\end{align}
establishing \eqref{eq6.13}.

To establish \eqref{eq6.14}, note from the second inequality in \eqref{eq6.12} that
\begin{equation}\label{eq6.15}
\lambda_2\left(\bigcup_{i\in\frakI(\sigma_0)}S^\dagger_{1,k}(i)\right)\ge\frac{1}{10dH}.
\end{equation}
It then follows from \eqref{eq6.3} with $\sigma=1$ and \eqref{eq6.15} that
\begin{align}
&
\lambda_2\left((\WWW\bigtriangleup\WWW_1)\cap\bigcup_{i\in\frakI(\sigma_0)}S^\dagger_{1,k}(i)\right)
\le\lambda_2((\WWW\bigtriangleup\WWW_1)\cap(\III_1\times\{y_1\}\times[0,1)))
\nonumber
\\
&\quad
<\frac{2\eps}{H}
\le20\eps d\lambda_2\left(\bigcup_{i\in\frakI(\sigma_0)}S^\dagger_{1,k}(i)\right).
\nonumber
\end{align}

This completes the proof.
\end{proof}

\begin{lemma}\label{lem65}
Let $x_1$ be the $x$-coordinate of a point of $S^\dagger_{1,k}(i)$ for some $i\in\frakI(\sigma_0)$.
Then
\textcolor{white}{xxxxxxxxxxxxxxxxxxxxxxxxxxxxxx}
\begin{displaymath}
N\beta+\cs_1(x_1;\WWW)=\cs_{\sigma_0}(x_{\sigma_0};\WWW),
\end{displaymath}
where
\textcolor{white}{xxxxxxxxxxxxxxxxxxxxxxxxxxxxxx}
\begin{equation}\label{eq6.16}
\Vert N\beta-u_0\Vert<\delta
\end{equation}
and $x_{\sigma_0}$ is the $x$-coordinate of a point of $T_{\alpha,\beta}^NS^\dagger_{1,k}(i)$
satisfying $\{x_{\sigma_0}\}=\{x_1+N\alpha\}$.
\end{lemma}

\begin{proof}[Sketch of proof]
Note that $T_{\alpha,\beta}^N$ gives rise to a shift in the $z$-direction by $N\beta\bmod{1}$
and a shift in the $x$-direction by $N\alpha\bmod{1}$.
Note also that $\WWW$ is invariant under $T_{\alpha,\beta}$.
We have earlier chosen $N$ to satisfy \eqref{eq6.16} and $\Vert N\alpha\Vert<\eps_1$.
\end{proof}

We are interested in the set
\begin{equation}\label{eq6.17}
\XXX_0=\left\{x:(x,y_1,z)\in\bigcup_{i\in\frakI(\sigma_0)}S^\dagger_{1,k}(i)\mbox{ for some $z\in[0,1)$}\right\}
\end{equation}
of the $x$-coordinates of points on $Y_1$ and the set
\begin{equation}\label{eq6.18}
\XXX_N=\left\{x:(x,y_{\sigma_0},z)\in\bigcup_{i\in\frakI(\sigma_0)}T_{\alpha,\beta}^NS^\dagger_{1,k}(i)\mbox{ for some $z\in[0,1)$}\right\}
\end{equation}
of the $x$-coordinates of points in~$Y_{\sigma_0}$.

Let $x_1\in\XXX_0$ and $x_{\sigma_0}\in\XXX_N$ be chosen arbitrarily.
Applying Lemma~\ref{lem61} with $\sigma=\sigma_0$ to the sets
\textcolor{white}{xxxxxxxxxxxxxxxxxxxxxxxxxxxxxx}
\begin{equation}\label{eq6.19}
U_1=\cs_1(x_1;\WWW_1)
\quad\mbox{and}\quad
U_{\sigma_0}=\cs_{\sigma_0}(x_{\sigma_0};\WWW_1),
\end{equation}
we deduce that
\textcolor{white}{xxxxxxxxxxxxxxxxxxxxxxxxxxxxxx}
\begin{displaymath}
\lambda_1(U_{\sigma_0}\bigtriangleup(u_0+U_1))\ge\frac{1}{32d^2}\lambda_1(U_1)(1-\lambda_1(U_1)).
\end{displaymath}
By the continuity of the Lebesgue measure, for any fixed $\eta>0$, if $\delta>0$ in \eqref{eq6.16} is sufficiently small, then
\begin{displaymath}
\lambda_1(U_{\sigma_0}\bigtriangleup(N\beta+U_1))\ge\frac{1}{32d^2}\lambda_1(U_1)(1-\lambda_1(U_1))-\eta.
\end{displaymath}
Choosing
\textcolor{white}{xxxxxxxxxxxxxxxxxxxxxxxxxxxxxx}
\begin{displaymath}
\eta=\frac{1}{64d^2}\lambda_1(U_1)(1-\lambda_1(U_1)),
\end{displaymath}
we deduce that
\textcolor{white}{xxxxxxxxxxxxxxxxxxxxxxxxxxxxxx}
\begin{equation}\label{eq6.20}
\lambda_1(U_{\sigma_0}\bigtriangleup(N\beta+U_1))\ge\frac{1}{64d^2}\lambda_1(U_1)(1-\lambda_1(U_1)),
\end{equation}
provided that $\delta>0$ is sufficiently small.

\begin{lemma}\label{lem66}
We have
\begin{equation}\label{eq6.21}
\lambda_1(\{x\in\XXX_N:\lambda_1(\cs_{\sigma_0}(x;\WWW)\bigtriangleup\cs_{\sigma_0}(x;\WWW_1))>300\eps d\})
<\frac{1}{10}\lambda_1(\XXX_N),
\end{equation}
provided that $\eps_1$ is sufficiently small.
We also have
\begin{equation}\label{eq6.22}
\lambda_1(\{x\in\XXX_0:\lambda_1(\cs_1(x;\WWW)\bigtriangleup\cs_1(x;\WWW_1))>200\eps d\})
<\frac{1}{10}\lambda_1(\XXX_0).
\end{equation}
\end{lemma}

\begin{proof}
First, note that
\begin{displaymath}
\int_{\XXX_N}\lambda_1(\cs_{\sigma_0}(x;\WWW)\bigtriangleup\cs_{\sigma_0}(x;\WWW_1))\,\dd x
=\lambda_2\left((\WWW\bigtriangleup\WWW_1)\cap\bigcup_{i\in\frakI(\sigma_0)}T_{\alpha,\beta}^NS^\dagger_{1,k}(i)\right).
\end{displaymath}
Combining this with \eqref{eq6.13} and noting
\begin{displaymath}
\lambda_2\left(\bigcup_{i\in\frakI(\sigma_0)}T_{\alpha,\beta}^NS^\dagger_{1,k}(i)\right)=\lambda_1(\XXX_N),
\end{displaymath}
we conclude that
\begin{displaymath}
\int_{\XXX_N}\lambda_1(\cs_{\sigma_0}(x;\WWW)\bigtriangleup\cs_{\sigma_0}(x;\WWW_1))\,\dd x
<2\eps_1+20\eps d\lambda_1(\XXX_N)
<30\eps d\lambda_1(\XXX_N)
\end{displaymath}
provided that
\textcolor{white}{xxxxxxxxxxxxxxxxxxxxxxxxxxxxxx}
\begin{equation}\label{eq6.23}
\eps_1<5\eps d\lambda_1(\XXX_N).
\end{equation}
The inequality \eqref{eq6.21} follows immediately from this.

Next, note that
\begin{displaymath}
\int_{\XXX_0}\lambda_1(\cs_1(x;\WWW)\bigtriangleup\cs_1(x;\WWW_1))\,\dd x
=\lambda_2\left((\WWW\bigtriangleup\WWW_1)\cap\bigcup_{i\in\frakI(\sigma_0)}S^\dagger_{1,k}(i)\right).
\end{displaymath}
Combining this with \eqref{eq6.14} and noting
\begin{displaymath}
\lambda_2\left(\bigcup_{i\in\frakI(\sigma_0)}S^\dagger_{1,k}(i)\right)=\lambda_1(\XXX_0),
\end{displaymath}
we conclude that
\textcolor{white}{xxxxxxxxxxxxxxxxxxxxxxxxxxxxxx}
\begin{displaymath}
\int_{\XXX_0}\lambda_1(\cs_1(x;\WWW)\bigtriangleup\cs_1(x;\WWW_1))\,\dd x
<20\eps d\lambda_1(\XXX_0).
\end{displaymath}
The inequality \eqref{eq6.22} follows immediately from this.
\end{proof}

It now follows from Lemma~\ref{lem66} that there exist $x_1^\star\in\XXX_0$ and $x_{\sigma_0}^\star\in\XXX_N$
satisfying $\{x_{\sigma_0}^\star\}=\{x_1^\star+N\alpha\}$ such that
\begin{equation}\label{eq6.24}
\lambda_1(\cs_{\sigma_0}(x_{\sigma_0}^\star;\WWW)\bigtriangleup\cs_{\sigma_0}(x_{\sigma_0}^\star;\WWW_1))\le300\eps d,
\end{equation}
provided that \eqref{eq6.23} holds, and
\begin{equation}\label{eq6.25}
\lambda_1(\cs_1(x_1^\star;\WWW)\bigtriangleup\cs_1(x_1^\star;\WWW_1))\le200\eps d.
\end{equation}

Combining \eqref{eq6.19} and \eqref{eq6.20}, we have
\begin{align}\label{eq6.26}
\frac{1}{64d^2}\lambda_1(U_1)(1-\lambda_1(U_1))
&
\le\lambda_1(\cs_{\sigma_0}(x_{\sigma_0};\WWW_1)\bigtriangleup(N\beta+\cs_1(x_1;\WWW_1)))
\nonumber
\\
&
=\lambda_1(\cs_{\sigma_0}(x_{\sigma_0}^\star;\WWW_1)\bigtriangleup(N\beta+\cs_1(x_1^\star;\WWW_1))),
\end{align}
where the last step is justified by Property~A.
Next, by the triangle inequality
\begin{displaymath}
\lambda_1(A\bigtriangleup C)\le\lambda_1(A\bigtriangleup B)+\lambda_1(B\bigtriangleup C),
\end{displaymath}
valid for measurable sets $A,B,C\subset[0,1)$, we have
\begin{align}\label{eq6.27}
&
\lambda_1(\cs_{\sigma_0}(x_{\sigma_0}^\star;\WWW_1)\bigtriangleup(N\beta+\cs_1(x_1^\star;\WWW_1)))
\nonumber
\\
&\quad
\le\lambda_1(\cs_{\sigma_0}(x_{\sigma_0}^\star;\WWW_1)\bigtriangleup\cs_{\sigma_0}(x_{\sigma_0}^\star;\WWW))
\nonumber
\\
&\quad\qquad
+\lambda_1(\cs_{\sigma_0}(x_{\sigma_0}^\star;\WWW)\bigtriangleup(N\beta+\cs_1(x_1^\star;\WWW)))
\nonumber
\\
&\quad\qquad
+\lambda_1((N\beta+\cs_1(x_1^\star;\WWW))\bigtriangleup(N\beta+\cs_1(x_1^\star;\WWW_1)))
\nonumber
\\
&\quad
\le\lambda_1(\cs_{\sigma_0}(x_{\sigma_0}^\star;\WWW_1)\bigtriangleup\cs_{\sigma_0}(x_{\sigma_0}^\star;\WWW))
\nonumber
\\
&\quad\qquad
+\lambda_1((N\beta+\cs_1(x_1^\star;\WWW))\bigtriangleup(N\beta+\cs_1(x_1^\star;\WWW_1))),
\end{align}
where the last step is justified by Lemma~\ref{lem65}.

Combining \eqref{eq6.24}--\eqref{eq6.27}, we now conclude that if \eqref{eq6.23} holds, then
\begin{equation}\label{eq6.28}
\frac{1}{64d^2}\lambda_1(U_1)(1-\lambda_1(U_1))\le500\eps d.
\end{equation}
It remains to specify the parameters $\eps$ and~$\eps_1$.

The assumption that the discrete transformation $T_{\alpha,\beta}$ is not ergodic also implies that
the $\bfv$-flow in $\MMM$ is not ergodic.
Thus Lemma~\ref{lem53} implies that
\begin{displaymath}
\frac{1}{d}\le\lambda_1(\cs_1(x_1^\star;\WWW))\le\frac{d-1}{d}.
\end{displaymath}
Combining this with \eqref{eq6.25}, we deduce that
\begin{displaymath}
\frac{1}{d}-200\eps d\le\lambda_1(\cs_1(x_1^\star;\WWW_1))\le\frac{d-1}{d}+200\eps d.
\end{displaymath}
Choosing $\eps$ to satisfy
\textcolor{white}{xxxxxxxxxxxxxxxxxxxxxxxxxxxxxx}
\begin{equation}\label{eq6.29}
200\eps d\le\frac{1}{2d},
\end{equation}
and noting that $U_1=\cs_1(x_1^\star;\WWW_1)$, we deduce that
\begin{displaymath}
\frac{1}{2d}\le\lambda_1(U_1)\le1-\frac{1}{2d}.
\end{displaymath}
Hence
\textcolor{white}{xxxxxxxxxxxxxxxxxxxxxxxxxxxxxx}
\begin{equation}\label{eq6.30}
\frac{1}{64d^2}\lambda_1(U_1)(1-\lambda_1(U_1))\ge\frac{1}{64d^2}\left(\frac{1}{2d}\right)^2.
\end{equation}
Clearly \eqref{eq6.28} and \eqref{eq6.30} contradict each other if
\begin{displaymath}
\eps<\frac{1}{128000d^5}.
\end{displaymath}
This is stronger than the restriction \eqref{eq6.29}.

Finally, note from \eqref{eq6.17} and \eqref{eq6.18} that
\begin{displaymath}
\lambda_1(\XXX_N)=\lambda_1(\XXX_0)\le\vert\III_1\vert=\vert\III_0\vert=\frac{1}{H},
\end{displaymath}
where $H=H(\WWW_1)$ is chosen sufficiently large.
Here $\WWW_1$ depends on~$\eps$, in view of \eqref{eq6.1}.
Clearly $\eps_1$ can be chosen sufficiently small in terms of $\eps$ so that the condition \eqref{eq6.23} is satisfied.

This gives ergodicity in Theorem~\ref{thm3}.
Unique ergodicity can then be established by a straightforward adaptation of the argument in Section~\ref{sec4}.

%
%

\section{Proof of Lemma~\ref{lem61}}\label{sec7}

For every $\sigma=1,\ldots,d$, let
\begin{displaymath}
\chi_{U_\sigma}(x)=\left\{\begin{array}{ll}
1,&\mbox{if $x\in U_\sigma$},\\
0,&\mbox{if $x\not\in U_\sigma$},
\end{array}\right.
\end{displaymath}
and
\textcolor{white}{xxxxxxxxxxxxxxxxxxxxxxxxxxxxxx}
\begin{displaymath}
\chi_{u+U_1}(x)=\left\{\begin{array}{ll}
1,&\mbox{if $x\in u+U_1$},\\
0,&\mbox{if $x\not\in u+U_1$},
\end{array}\right.
\end{displaymath}
denote respectively the characteristic functions of $U_\sigma$ and~$u+U_1$.
Then it is not difficult to see that
\begin{equation}\label{eq7.1}
\lambda_1(U_\sigma\bigtriangleup(u_0+U_1))=\int_0^1(\chi_{U_\sigma}(x)-\chi_{u+U_1}(x))^2\,\dd x.
\end{equation}
On the other hand, the $n$-th Fourier coefficient of $\chi_{U_\sigma}$ and of $\chi_{u+U_1}$ are respectively
\begin{displaymath}
a_{n;\sigma}=\int_0^1\chi_{U_\sigma}(x)\ee^{-2\pi\ii nx}\,\dd x
\end{displaymath}
and
\textcolor{white}{xxxxxxxxxxxxxxxxxxxxxxxxxxxxxx}
\begin{displaymath}
\int_0^1\chi_{u+U_1}(x)\ee^{-2\pi\ii nx}\,\dd x
=\ee^{-2\pi\ii nu}\int_0^1\chi_{U_1}(x-u)\ee^{-2\pi\ii n(x-u)}\,\dd x
=\ee^{-2\pi\ii nu}a_{n;1}.
\end{displaymath}
Applying the Parseval formula to the right hand side of \eqref{eq7.1} then gives
\begin{equation}\label{eq7.2}
\lambda_1(U_\sigma\bigtriangleup(u_0+U_1))=\sum_{n\in\Zz}\vert a_{n;\sigma}-\ee^{-2\pi\ii nu}a_{n;1}\vert^2.
\end{equation}

Suppose that $n\in\Zz$ is non-zero.
The complex valued function $\ee^{-2\pi\ii nu}a_{n;1}$ is periodic in $u\in\Rr$ with period $\vert n\vert^{-1}$.
Indeed, its values over a period forms a circle
\begin{displaymath}
\Circle(n;1)=\{\ee^{-2\pi\ii nu}a_{n;1}:0\le u<\vert n\vert^{-1}\}
\end{displaymath}
on the complex plane, as shown on the left in Figure~7.1.

\begin{displaymath}
\begin{array}{c}
\includegraphics[scale=0.8]{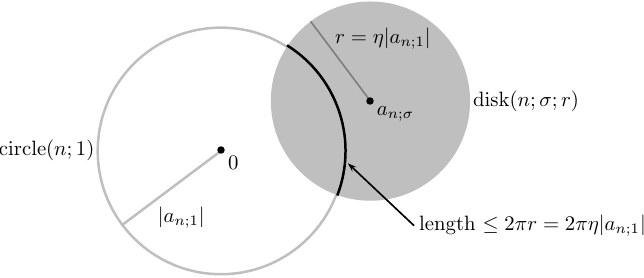}
\vspace{3pt}\\
\mbox{Figure 7.1: $\Circle(n;1)$ and $\Disk(n;\sigma;r)$}
\end{array}
\end{displaymath}

Let $\Disk(n;\sigma;r)$ denote the circular disk on the complex plane centred at $a_{n;\sigma}$ and with radius $r>0$,
as shown on the right in Figure~7.1.
The intersection of $\Circle(n;1)$ and $\Disk(n;\sigma;r)$, if not empty, is a circular arc within $\Disk(n;\sigma;r)$ with length clearly
less than the circumference $2\pi r$ of the disk.

We let $r=\eta\vert a_{n;1}\vert$, where the parameter $\eta\in(0,1)$ will be determined later.
Now partition $\Circle(n;1)$ into two parts, one disjoint from $\Disk(n;\sigma;\eta\vert a_{n;1}\vert)$ and the other contained
in $\Disk(n;\sigma;\eta\vert a_{n;1}\vert)$.
It is then clear that the inequality
\begin{equation}\label{eq7.3}
\vert a_{n;\sigma}-\ee^{-2\pi\ii nu}a_{n;1}\vert^2\ge\eta^2\vert a_{n;1}\vert^2
\end{equation}
holds for $u\in[0,1)$, with the possible exception of a subset
\begin{equation}\label{eq7.4}
\bad(n;\sigma;\eta)\subset[0,1)
\quad\mbox{with}\quad
\lambda_1(\bad(n;\sigma;\eta))\le\eta.
\end{equation}

For every $n\in\Zz\setminus\{0\}$, every $\sigma=1,\ldots,d$ and every $u\in[0,1)$, let
\begin{equation}\label{eq7.5}
F_{n;\sigma;\eta}(u)=\left\{\begin{array}{ll}
\vert a_{n;1}\vert^2,&\mbox{if $u\in\bad(n;\sigma;\eta)$},\\
0,&\mbox{if $u\not\in\bad(n;\sigma;\eta)$}.
\end{array}\right.
\end{equation}
Then for every $\sigma=1,\ldots,d$, it follows from \eqref{eq7.4} and \eqref{eq7.5} that
\begin{equation}\label{eq7.6}
\sum_{n\in\Zz\setminus\{0\}}\int_0^1F_{n;\sigma;\eta}(u)\,\dd u\le\eta\sum_{n\in\Zz\setminus\{0\}}\vert a_{n;1}\vert^2.
\end{equation}
If, further, we let
\begin{equation}\label{eq7.7}
\violator(\sigma;\eta)=\left\{u\in[0,1):\sum_{n\in\Zz\setminus\{0\}}F_{n;\sigma;\eta}(u)\ge\frac{1}{2}\sum_{n\in\Zz\setminus\{0\}}\vert a_{n;1}\vert^2\right\},
\end{equation}
then it follows from \eqref{eq7.5} and \eqref{eq7.6} that
\begin{equation}\label{eq7.8}
\lambda_1(\violator(\sigma;\eta))\le2\eta.
\end{equation}
If $u\in[0,1)\setminus\violator(\sigma;\eta)$, then using \eqref{eq7.5} and \eqref{eq7.7}, we have
\begin{align}\label{eq7.9}
\sum_{\substack{{n\in\Zz\setminus\{0\}}\\{F_{n;\sigma;\eta}(u)=0}}}\vert a_{n;1}\vert^2
&
=\sum_{n\in\Zz\setminus\{0\}}\vert a_{n;1}\vert^2
-\sum_{\substack{{n\in\Zz\setminus\{0\}}\\{F_{n;\sigma;\eta}(u)\ne0}}}\vert a_{n;1}\vert^2
\nonumber
\\
&
=\sum_{n\in\Zz\setminus\{0\}}\vert a_{n;1}\vert^2
-\sum_{n\in\Zz\setminus\{0\}}F_{n;\sigma;\eta}(u)
>\frac{1}{2}\sum_{n\in\Zz\setminus\{0\}}\vert a_{n;1}\vert^2.
\end{align}
Recalling that the condition $F_{n;\sigma;\eta}(u)=0$ implies the inequality \eqref{eq7.3}, it then follows from \eqref{eq7.9} that
\begin{displaymath}
\sum_{\substack{{n\in\Zz\setminus\{0\}}\\{F_{n;\sigma;\eta}(u)=0}}}\vert a_{n;\sigma}-\ee^{-2\pi\ii nu}a_{n;1}\vert^2
\ge\eta^2\sum_{\substack{{n\in\Zz\setminus\{0\}}\\{F_{n;\sigma;\eta}(u)=0}}}\vert a_{n;1}\vert^2
>\frac{\eta^2}{2}\sum_{n\in\Zz\setminus\{0\}}\vert a_{n;1}\vert^2.
\end{displaymath}
We therefore deduce that for every $\sigma=1,\ldots,d$, the inequality
\begin{equation}\label{eq7.10}
\sum_{n\in\Zz\setminus\{0\}}\vert a_{n;\sigma}-\ee^{-2\pi\ii nu}a_{n;1}\vert^2
>\frac{\eta^2}{2}\sum_{n\in\Zz\setminus\{0\}}\vert a_{n;1}\vert^2
\end{equation}
holds for every $u\in[0,1)\setminus\violator(\sigma;\eta)$, where the bound \eqref{eq7.8} holds.

To evaluate the sum
\textcolor{white}{xxxxxxxxxxxxxxxxxxxxxxxxxxxxxx}
\begin{equation}\label{eq7.11}
\sum_{n\in\Zz\setminus\{0\}}\vert a_{n;1}\vert^2
=\sum_{n\in\Zz}\vert a_{n;1}\vert^2-\vert a_{0;1}\vert^2,
\end{equation}
we note that the Parseval formula gives
\begin{equation}\label{eq7.12}
\sum_{n\in\Zz}\vert a_{n;1}\vert^2=\int_0^1\chi_{U_1}^2(x)\,\dd x=\int_0^1\chi_{U_1}(x)\,\dd x=\lambda_1(U_1),
\end{equation}
while by definition we have
\begin{equation}\label{eq7.13}
\vert a_{0;1}\vert^2=\left(\int_0^1\chi_{U_1}(x)\,\dd x\right)^2=(\lambda_1(U_1))^2.
\end{equation}
From \eqref{eq7.11}--\eqref{eq7.13}, we deduce that
\begin{equation}\label{eq7.14}
\sum_{n\in\Zz\setminus\{0\}}\vert a_{n;1}\vert^2=\lambda_1(U_1)(1-\lambda_1(U_1)).
\end{equation}

Combining \eqref{eq7.2}, \eqref{eq7.10} and \eqref{eq7.14}, we now deduce that for every $\sigma=1,\ldots,d$, the inequality
\textcolor{white}{xxxxxxxxxxxxxxxxxxxxxxxxxxxxxx}
\begin{displaymath}
\lambda_1(U_\sigma\bigtriangleup(u_0+U_1))
>\frac{\eta^2}{2}\lambda_1(U_1)(1-\lambda_1(U_1))
\end{displaymath}
holds for every $u\in[0,1)\setminus\violator(\sigma;\eta)$, where the bound \eqref{eq7.8} holds.
Finally, let $\eta=(4d)^{-1}$.
Then it follows from \eqref{eq7.8} that
\begin{displaymath}
\lambda_1\left(\bigcup_{\sigma=1}^d\violator(\sigma;\eta)\right)\le2\eta d=\frac{1}{2}.
\end{displaymath}
This completes the proof of Lemma~\ref{lem61}.

%
%

\section{Kronecker--Weyl polycube translation $3$-manifolds}\label{sec8}

Both Theorems \ref{thm2} and~\ref{thm3} concern finite polycube translation $3$-manifolds that are the cartesian products
of finite polysquare translation surfaces with the unit torus $[0,1)^2$.
These special polycube translation $3$-manifolds all have one direction which is integrable,
and can therefore be considered as degenerate cases in the class of all finite polycube translation $3$-manifolds.

We are currently unable to answer in full generality the question raised in the Open Problem in Section~\ref{sec1}.
However, we are nevertheless able to answer this question in the affirmative for infinitely many finite polycube translation $3$-manifolds
that do not have any direction which can be considered integrable.
The construction of such examples of finite polycube translation $3$-manifolds comes from the notion of \textit{split-covering}.

Before we describe this construction, we consider a simple example.

\begin{example}\label{ex81}
Consider the unit torus $[0,1)^3$ as a basic building block.
Then any half-infinite geodesic in $[0,1)^3$ with a Kronecker direction is uniformly distributed.
The L-solid translation $3$-manifold $\MMM_0$ can be viewed as a $3$-fold split-covering of $[0,1)^3$,
and Theorem~\ref{thm1} states that any half-infinite geodesic in $\MMM_0$ with a Kronecker direction
is uniformly distributed unless it hits a singularity.

Let us now take $\MMM_0$ as a basic building block, and consider next a rather special $4$-fold split-covering of $\MMM_0$
where the $4$ copies of $\MMM_0$ are placed side-by-side in the $x$-direction, as shown in Figure~8.1.

\begin{displaymath}
\begin{array}{c}
\includegraphics[scale=0.8]{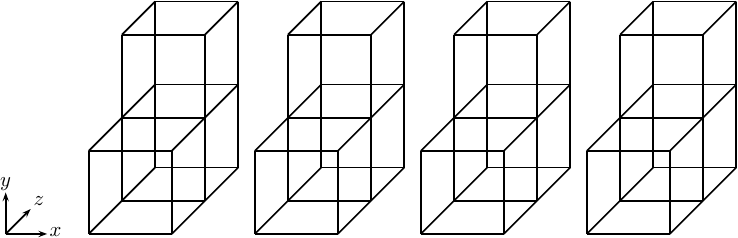}
\vspace{3pt}\\
\mbox{Figure 8.1: $4$ L-solid translation $3$-manifolds}
\end{array}
\end{displaymath}

Clearly if we simply glue the $4$ copies of $\MMM_0$ together, then the resulting polycube translation $3$-manifold $\MMM$
is integrable in the $x$-direction, and can be considered degenerate.
To ensure that the resulting polycube translation $3$-manifold $\MMM$ is not degenerate, we convert some of the $X$-faces into walls,
as shown in the examples in Figure~8.2.

\begin{displaymath}
\begin{array}{c}
\includegraphics[scale=0.8]{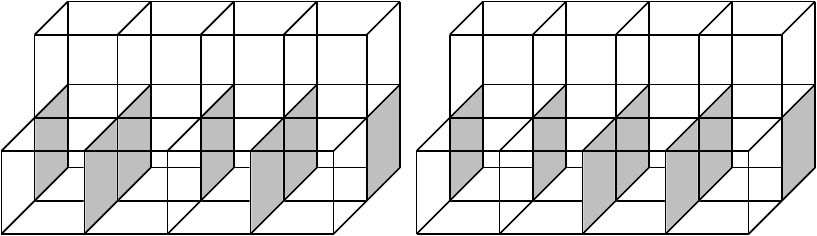}
\vspace{3pt}\\
\mbox{Figure 8.2: non-degenerate $4$-fold split-coverings}
\\
\mbox{of the L-solid translation $3$-manifolds}
\end{array}
\end{displaymath}

We claim that any half-infinite geodesic in either example of $\MMM$ in Figure~8.2 with a Kronecker direction
is uniformly distributed unless it hits a singularity.

To give a brief glimpse of our method and help formulate our result, let us consider the example of $\MMM$ on the right hand side,
and ignore the atomic cubes at the front, as indicated in Figure~8.3.

\begin{displaymath}
\begin{array}{c}
\includegraphics[scale=0.8]{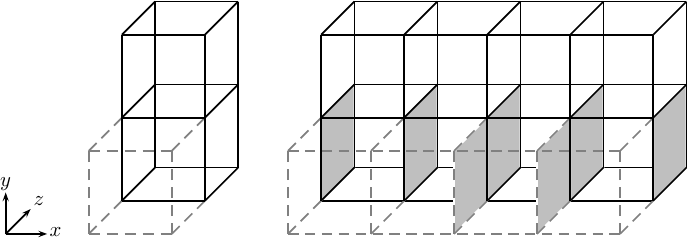}
\vspace{3pt}\\
\mbox{Figure 8.3: ignoring parts of a non-degenerate $4$-fold split-covering}
\\
\mbox{of the L-solid translation $3$-manifolds}
\end{array}
\end{displaymath}

Removing the front atomic cubes altogether, we obtain the collection of atomic cubes as shown in Figure~8.4.
Here we see a $4$-fold split-covering in the $x$-direction with an $X$-street of length~$4$,
while immediately below are $4$ $X$-streets of length~$1$.

\begin{displaymath}
\begin{array}{c}
\includegraphics[scale=0.8]{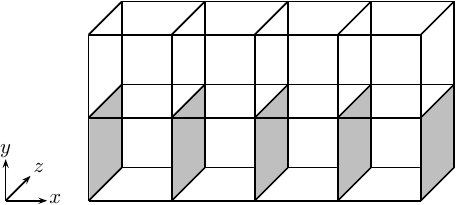}
\vspace{3pt}\\
\mbox{Figure 8.4: a very special $4$-fold split-covering}
\end{array}
\end{displaymath}

We shall show that this is the crucial part of~$\MMM$, and that the remaining atomic cubes of $\MMM$ are totally irrelevant.
\end{example}

\begin{theorem}\label{thm4}
Let $\PPP$ be a finite polysquare translation surface, and let $\MMM_0=[0,1)\times\PPP$ denote the cartesian product of the unit torus $[0,1)$
with~$\PPP$, where the $x$-direction is the direction of the unit torus.
For any fixed positive integer~$s$, let $\MMM$ denote an $s$-fold split-covering of~$\MMM_0$,
where the $s$ copies of $\MMM_0$ are placed side-by-side in the $x$-direction and glued together,
and where some $X$-faces are replaced by walls.

Suppose that $\MMM$ contains an $s\times2\times1$ array of atomic cubes, where the top row of $s$ atomic cubes
gives rise to an $X$-street of length~$s$, and where any $X$-face of any atomic cube in the bottom row is a wall.
Then any half-infinite geodesic in $\MMM$ with a Kronecker direction is uniformly distributed unless it hits a singularity.
\end{theorem}

\begin{remark}
Note that in Figure~8.4, we can take the bottom row of atomic cubes and place them at the top instead.
Note that in Figure~8.5, we must ensure that the bottom $Y$-face of any atomic cube in the bottom row is identified with the top $Y$-face
of the corrresponding atomic cube in the top row.

\begin{displaymath}
\begin{array}{c}
\includegraphics[scale=0.8]{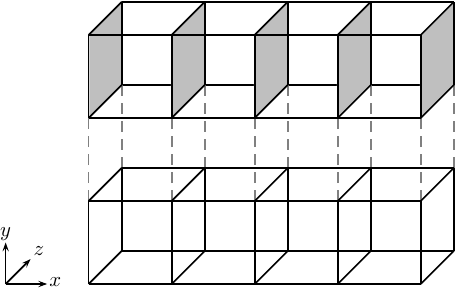}
\vspace{3pt}\\
\mbox{Figure 8.5: an alternative very special $4$-fold split-covering}
\end{array}
\end{displaymath}
\end{remark}

\begin{remark}
In Theorem~\ref{thm4}, we can also allow $\MMM$ to contain an $s\times2\times1$ array of atomic cubes, where the bottom row of $s$ atomic cubes
gives rise to an $X$-street of length~$s$, and where any $X$-face of any atomic cube in the top row is a wall.
\end{remark}

\begin{example}\label{ex82}
Consider a finite polysquare translation surface $\PPP$ of $9$ faces as shown in Figure~8.6,
where identified edges are obtained from each other by perpendicular translation,
and where $2$ of the $y$-perpendicular edges and $2$ of the $z$-perpendicular edges are walls.

\begin{displaymath}
\begin{array}{c}
\includegraphics[scale=0.8]{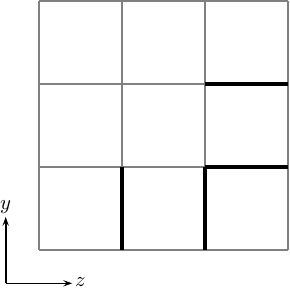}
\vspace{3pt}\\
\mbox{Figure 8.6: a finite polysquare translation surface}
\end{array}
\end{displaymath}

Let $\MMM_0=[0,1)\times\PPP$ denote the polycube translation $3$-manifold obtained from the cartesian product of the unit torus $[0,1)$
with~$\PPP$, as shown in the picture on the left in Figure~8.7.
Corresponding to the $2$ $y$-perpendicular edges of~$\PPP$, the $3$-manifold $\MMM_0$ has $2$ $Y$-faces that are walls.
Corresponding to the $2$ $z$-perpendicular edges of~$\PPP$, the $3$-manifold $\MMM_0$ has $2$ $Z$-faces that are walls.

\begin{displaymath}
\begin{array}{c}
\includegraphics[scale=0.8]{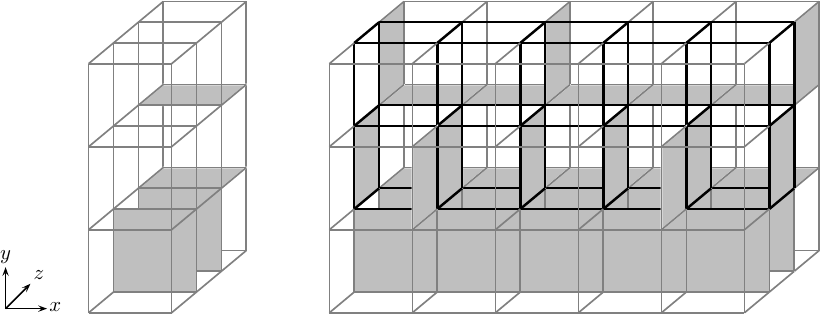}
\vspace{3pt}\\
\mbox{Figure 8.7: a very special $5$-fold split-covering}
\end{array}
\end{displaymath}

Next, we construct a $5$-fold split-covering of $\MMM_0$ with $5$ copies of $\MMM_0$ placed side-by-side in the $x$-direction and glued together,
and where some $X$-faces are replaced by walls, as shown in the picture on the right in Figure~8.7.
It is clear from this picture that the resulting polycube translation $3$-manifold $\MMM$ contains a $5\times2\times1$ array of atomic cubes
of the type described in the statement of Theorem~\ref{thm4}.
Those $X$-faces that are walls and are not in the special $5\times2\times1$ array of atomic cubes are irrelevant.
The $Y$-faces and $Z$-faces present in $\MMM$ correspond precisely to those present in~$\MMM_0$.
\end{example}

We say that a finite polycube translation $3$-manifold $\MMM$ is Kronecker--Weyl
if every half-infinite geodesic in $\MMM$ with a Kronecker direction is uniformly distributed unless it hits a singularity.
Using $s$-fold split-covering and noting the irrelevance of whether many $X$-faces are walls or otherwise,
we see that we can construct exponentially many such polycube translation $3$-manifolds.
On the other hand, we are a long way from a full solution of the Open Problem in Section~\ref{sec1}.

%
%

\section{Proof of Theorem~\ref{thm4}}\label{sec9}

The proof of Theorem~\ref{thm4} goes very much along the lines of that of Theorem~\ref{thm1}.
However, extra ideas are needed at various stages of the argument.

Suppose that the period polycube translation $3$-manifold $\MMM_0$ contains precisely $d$ distinct atomic cubes.
Let $Y^{(1)},\ldots,Y^{(d)}$ denote the bottom $Y$-faces of these atomic cubes, and let
\textcolor{white}{xxxxxxxxxxxxxxxxxxxxxxxxxxxxxx}
\begin{displaymath}
\YYY_0=Y^{(1)}\cup\ldots\cup Y^{(d)}.
\end{displaymath}
Then the $s$-fold split-covering $\MMM$ contains precisely $sd$ distinct atomic cubes, where each $Y$-face $Y^{(\gamma)}$, $\gamma=1,\ldots,d$,
of $\MMM_0$ leads to $s$ distinct $Y$-faces $Y^{(\gamma)}_\sigma$, $\sigma=1,\ldots,s$.
Let
\textcolor{white}{xxxxxxxxxxxxxxxxxxxxxxxxxxxxxx}
\begin{displaymath}
\YYY=\bigcup_{\gamma=1}^d\bigcup_{\sigma=1}^sY^{(\gamma)}_\sigma.
\end{displaymath}

As before, we consider the discretization of the $\bfv$-flow in $\MMM$ relative to the $Y$-faces,
where $\bfv=(\alpha,1,\beta)$ is a Kronecker direction, and let
\begin{displaymath}
T_{\alpha,\beta}:\YYY\to\YYY
\end{displaymath}
denote the relevant discrete transformation defined by consecutive hitting points.
We need to show that this area-preserving map is ergodic on~$\YYY$.
Suppose, on the contrary, that it is not ergodic.
Then there exist two disjoint $T_{\alpha,\beta}$-invariant subsets $\WWW$ and $\SSS$ of $\YYY$
such that $\YYY=\WWW\cup\SSS$ and
\begin{equation}\label{eq9.1}
0<\lambda_2(\WWW)\le\lambda_2(\SSS)<sd,
\end{equation}
where $\lambda_2$ denotes the $2$-dimensional Lebesgue measure.
We show that this leads to a contradiction.

The natural projection of $\YYY$ to $\YYY_0$ can be represented in the form
\begin{equation}\label{eq9.2}
\YYY\to\YYY_0:(x,y,z)\mapsto(\{x\},y,z),
\end{equation}
indicating that, for every $\gamma=1,\ldots,d$, the images of the $Y$-faces $Y^{(\gamma)}_\sigma$, $\sigma=1,\ldots,s$,
coincide and are all equal to~$Y^{(\gamma)}$,
and the image of each point in $Y^{(\gamma)}_\sigma$ is in the same relative position in~$Y^{(\gamma)}$.
Thus for every point $P\in\YYY_0$, there are precisely $s$ distinct points $P_1,\ldots,P_s\in\YYY$
which have projection image~$P$.
Let
\begin{displaymath}
f_\WWW(P)=\vert\{P_1,\ldots,P_s\}\cap\WWW\vert
\quad\mbox{and}\quad
f_\SSS(P)=\vert\{P_1,\ldots,P_s\}\cap\SSS\vert
\end{displaymath}
denote the number of these $s$ points that fall into $\WWW$ and $\SSS$ respectively.
They are non-negative integer valued functions defined on $\YYY_0$ that satisfy the condition
\begin{equation}\label{eq9.3}
f_\WWW(P)+f_\SSS(P)=s
\quad\mbox{for almost all $P\in\YYY_0$}.
\end{equation}
We also consider the corresponding projection of
\begin{equation}\label{eq9.4}
T_{\alpha,\beta}:\YYY\to\YYY
\quad\mbox{to}\quad
T_0=T_{0,\alpha,\beta}:\YYY_0\to\YYY_0,
\end{equation}
which is simply the discretization of the $\bfv$-flow in $\MMM_0$ relative to the $Y$-faces.

Theorem~\ref{thm3} establishes the ergodicity of~$T_0$.
The Birkhoff ergodic theorem then guarantees that both $f_\WWW$ and $f_\SSS$ are constant integer valued functions.
It then follows from \eqref{eq9.1} and \eqref{eq9.3} that there exist two integers $s_1$ and $s_2$ satisfying
\begin{equation}\label{eq9.5}
1\le s_1,s_2\le s-1
\quad\mbox{and}\quad
s_1+s_2=s
\end{equation}
such that
\textcolor{white}{xxxxxxxxxxxxxxxxxxxxxxxxxxxxxx}
\begin{equation}\label{eq9.6}
f_\WWW(P)=s_1
\quad\mbox{and}\quad
f_\SSS(P)=s_2
\quad\mbox{for almost all $P\in\YYY_0$}.
\end{equation}
We then derive a contradiction by establishing a suitable analogue of Lemma~\ref{lem12}.

For a Kronecker direction $\bfv=(\alpha,1,\beta)$, where $\alpha$, $\beta$ and $\alpha/\beta$ are all irrational,
we can clearly assume that $\alpha>0$ and $\beta>0$.
In the simple case of Theorem~\ref{thm1} and Lemma~\ref{lem12}, the parameter $\beta$ is concerned with a direction
which is considered to be integrable, so its value causes no difficulty.
Furthermore, since we can interchange the other two directions, we can assume that $0<\alpha<1$.

Here, the parameter $\beta$ is no longer concerned with a direction which is considered to be integrable,
so its value can potentially cause difficulties.
Meanwhile, the value of the parameter $\alpha$ can also cause some inconvenience.
To formulate and establish a suitable analogue of Lemma~\ref{lem12}, we need to first understand these difficulties.

\begin{remarka}
Assume that the $z$-direction is integrable, so that the parameter $\beta$ does not cause any difficulties.
Figure~9.1 shows the singularities of the $\bfv$-flow, in the special case when $0<\alpha<1$, in a $4\times2\times1$ array of atomic cubes
of the type described in the statement of Theorem~\ref{thm4} in connection with a $4$-fold split-covering, with the $z$-direction suppressed.
The dashed line segment in each bottom atomic cube indicates that a line segment in the $z$-direction on the bottom $Y$-face
of the atomic cube is taken by the flow to the singular edge at the top of the wall that is the right $X$-face of the atomic cube.

\begin{displaymath}
\begin{array}{c}
\includegraphics[scale=0.8]{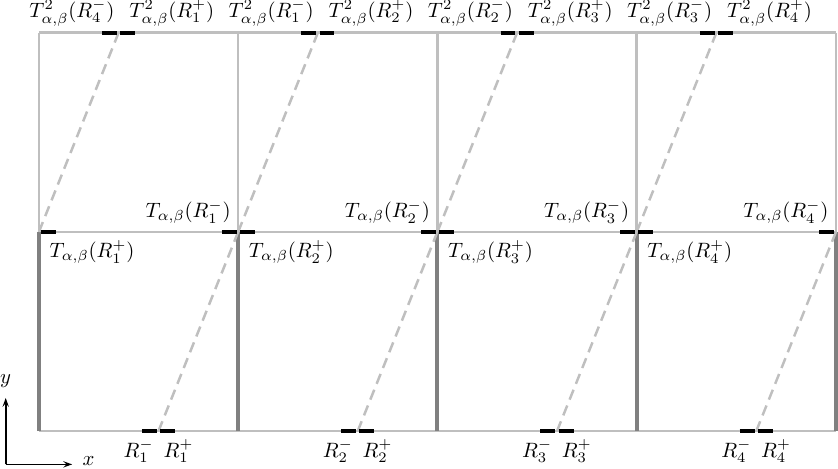}
\vspace{3pt}\\
\mbox{Figure 9.1: the case when $0<\alpha<1$}
\end{array}
\end{displaymath}

The situation is a bit more complicated when $\alpha>1$, as illustrated in Figure~9.2.

\begin{displaymath}
\begin{array}{c}
\includegraphics[scale=0.8]{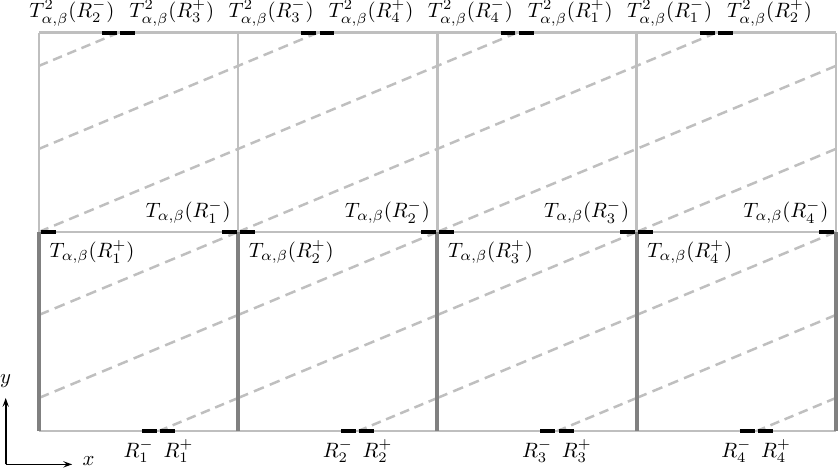}
\vspace{3pt}\\
\mbox{Figure 9.2: the case when $2<\alpha<3$}
\end{array}
\end{displaymath}

In each case, we start with $4$ small rectangles
\begin{displaymath}
R_1=R_1^-\cup R_1^+,
\quad
R_2=R_2^-\cup R_2^+,
\quad
R_3=R_3^-\cup R_3^+,
\quad
R_4=R_4^-\cup R_4^+.
\end{displaymath}
Each is on the bottom $Y$-face of a bottom atomic cube, and they are in the same relative position within their own $Y$-faces.
Each is split under $T_{\alpha,\beta}$ by the singular edge at the top of the wall that forms the right $X$-face of its own atomic cube.

Although we draw the same conclusion, for instance, that the image $T_{\alpha,\beta}^2(R_1^-)$ of the left half $R_1^-$ of $R_1$
and the image $T_{\alpha,\beta}^2(R_2^+)$ of the right half $R_2^+$ of $R_2$ together essentially form a small rectangle on some $Y$-face,
the identity of this $Y$-face is not so straightforward as in the case $0<\alpha<1$ earlier.

This is not a major issue, but we shall fix this at the same time as we deal with the real issue concerning the value of~$\beta$.
The idea is that of \textit{short transportation}, illustrated in Figure~9.3, where we keep our attention
on the flow much closer to the singularities.

\begin{displaymath}
\begin{array}{c}
\includegraphics[scale=0.8]{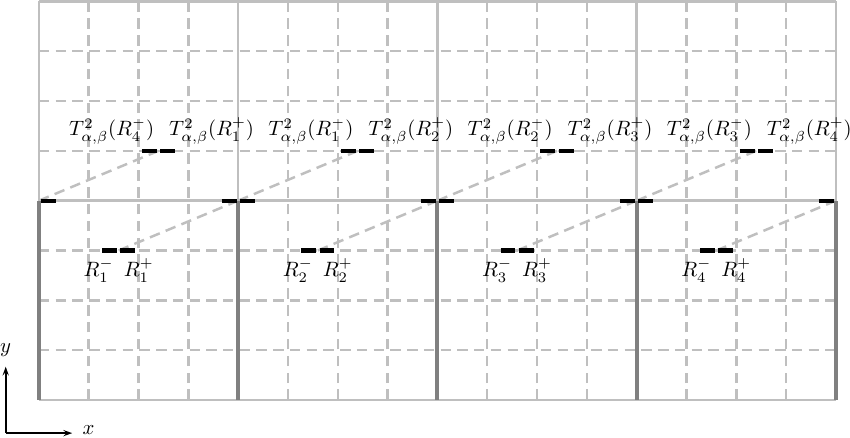}
\vspace{3pt}\\
\mbox{Figure 9.3: fixing the problem with the value of $\alpha$}
\end{array}
\end{displaymath}
\end{remarka}

\begin{remarkb}
The parameter $\beta$ deals with the part of the flow in the $z$-direction.
Figure~9.1 shows the singularities of the $\bfv$-flow, in the special case when $\beta>0$ is small, in a $4\times2\times1$ array of atomic cubes
of the type described in the statement of Theorem~\ref{thm4} in connection with a $4$-fold split-covering.
The dashed line segment in each bottom atomic cube indicates that a line segment in the $z$-direction on the bottom $Y$-face
of the atomic cube is taken by the flow to the singular edge at the top of the wall that is the right $X$-face of the atomic cube.
As the value of $\beta$ is small, we can find suitable small rectangles such that their images under $T_{\alpha,\beta}$ and $T_{\alpha,\beta}^2$
stay within this $4\times2\times1$ array of atomic cubes.

\begin{displaymath}
\begin{array}{c}
\includegraphics[scale=0.8]{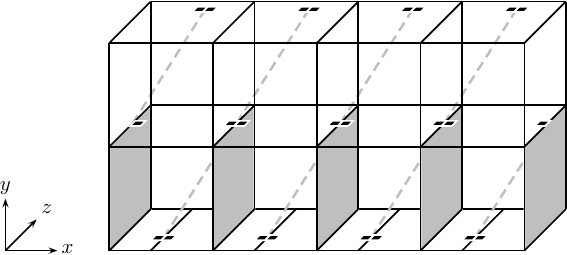}
\vspace{3pt}\\
\mbox{Figure 9.4: small values of $\beta$}
\end{array}
\end{displaymath}

The situation is different if $\beta>0$ is not so small, as there is the possibility that the flow from a small rectangle
crosses the back $Z$-face of the bottom atomic cube before reaching the wall that is the right $X$-face of the atomic cube,
as illustrated in Figure~9.5.
In this case, the images of the small rectangle under $T_{\alpha,\beta}$ and $T_{\alpha,\beta}^2$
may no longer stay within this $4\times2\times1$ array of atomic cubes.
What then happens depends on whether some $X$-faces outside this array are walls or not.

\begin{displaymath}
\begin{array}{c}
\includegraphics[scale=0.8]{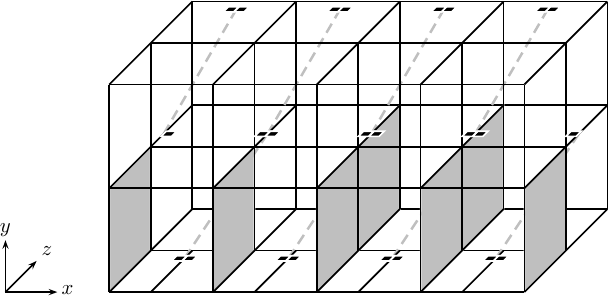}
\vspace{3pt}\\
\mbox{Figure 9.5: larger values of $\beta$}
\end{array}
\end{displaymath}

Figure~9.6 shows the view from the top of what may happen, with small $\alpha>0$,
and we do not have the appropriate composition of the image rectangles.
For instance, the images $T_{\alpha,\beta}^2(R_1^-)$ and $T_{\alpha,\beta}^2(R_1^+)$ now form a small rectangle.
It seems complicated to eradicate the situation if $\beta>1$.

\begin{displaymath}
\begin{array}{c}
\includegraphics[scale=0.8]{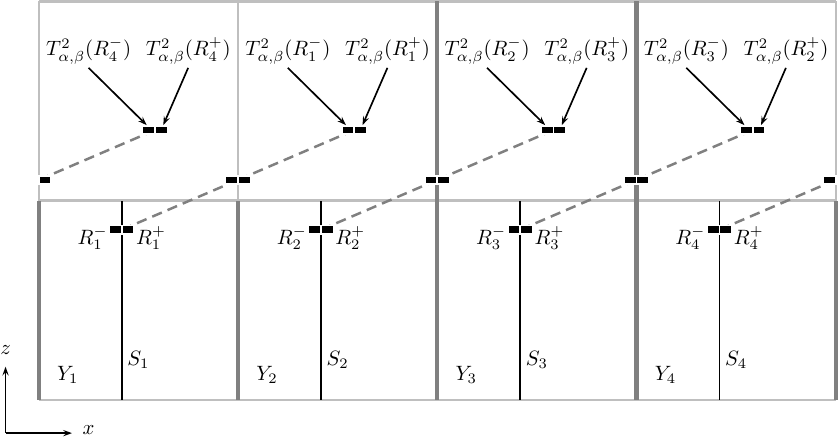}
\vspace{3pt}\\
\mbox{Figure 9.6: larger values of $\beta$}
\end{array}
\end{displaymath}
\end{remarkb}

Figure~9.3 gives a hint on a possible remedy of the situation for the value of the parameter $\alpha$ if the value of the parameter $\beta$
does not cause any serious difficulties.
It turns out that essentially the same idea can also deal with any difficulties arising from the value of the parameter~$\beta$.

Short transportation means concentrating on what happens much closer to the singularities.
However, to do so, we need to move into the interior of the atomic cubes and not just the $Y$-faces.
Thus we move outside the set $\YYY$ and have to abandon the discretization of the $\bfv$-flow in $\MMM$ relative to the $Y$-faces.

A simpler approach is to magnify the polycube translation $3$-manifold $\MMM$ instead.
More precisely, given any Kronecker direction $\bfv=(\alpha,1,\beta)$ with $\alpha>0$ and $\beta>0$, let $\NNNN$ be an integer satisfying
\begin{equation}\label{eq9.7}
\NNNN>\max\{\alpha,1+2\beta\}.
\end{equation}
We now magnify the polycube translation $3$-manifold $\MMM$ by a factor $\NNNN$ in each direction, so that any atomic cube in $\MMM$
is magnified to volume~$\NNNN^3$, and any face in $\MMM$ is magnified to area~$\NNNN^2$.
We then split each magnified atomic cube into $\NNNN^3$ atomic cubes of unit volume,
so that each magnified face of $\MMM$ is split into $\NNNN^2$ faces of unit area.
We shall abuse notation and denote the resulting polycube translation $3$-manifold also by~$\MMM$.
This new $\MMM$ has $\NNNN^3sd$ atomic cubes, with $\NNNN^3sd$ distinct $Y$-faces that make up a new set~$\YYY$.
We can perform a similar exercise to obtain a new polycube translation $3$-manifold $\MMM_0$ with $\NNNN^3d$ atomic cubes,
with $\NNNN^3d$ distinct $Y$-faces that make up a corresponding new set~$\YYY_0$.
Clearly the new $\MMM$ is an $s$-fold split-covering of the new~$\MMM_0$.

We can consider the corresponding new mappings given in \eqref{eq9.4}, and repeat our previous argument.
Then the inequalities \eqref{eq9.1} are replaced by the inequalities
\begin{displaymath}
0<\lambda_2(\WWW)\le\lambda_2(\SSS)<\NNNN^3sd.
\end{displaymath}
Crucially, the assertions \eqref{eq9.5} and \eqref{eq9.6} remain valid.

The special $s\times2\times1$ array of atomic cubes is now replaced by an $s\NNNN\times2\NNNN\times\NNNN$ array of atomic cubes.
Suppose that the coordinates of the points are in
\begin{equation}\label{eq9.8}
[0,s\NNNN]\times[0,2\NNNN]\times[0,\NNNN].
\end{equation}
Then the walls correspond precisely to the collection of points
\begin{displaymath}
\{(x,y,z):x\in\{0,\NNNN,\ldots,s\NNNN\}\mbox{ and }y,z\in[0,\NNNN]\}.
\end{displaymath}
For the special case when $s=4$ and $\NNNN=4$, this array is shown in Figure~9.3 with the $z$-direction suppressed.

It is also convenient to write
\begin{equation}\label{eq9.9}
\BBBB_\sigma=[(\sigma-1)\NNNN,\sigma\NNNN]\times[0,2\NNNN]\times[0,\NNNN],
\quad
\sigma=1,\ldots,s.
\end{equation}
Their union gives the $s$-fold split-covering of $[0,\NNNN]\times[0,2\NNNN]\times[0,\NNNN]\subset\MMM_0$
that corresponds to $\MMM$ being the $s$-fold split-covering of~$\MMM_0$.

We are ready to formulate a suitable analogue of Lemma~\ref{lem12}.
To do so, we need to find small special rectangles $R_1,\ldots,R_s$ on suitable Y-faces $Y_1,\ldots,Y_s$
as well as small special rectangles $R^*_1,\ldots,R^*_s$ on suitable Y-faces $Y^*_1,\ldots,Y^*_s$.

\begin{remark}
We comment in advance that while in Lemma~\ref{lem12}, we have
\begin{displaymath}
\YYY=Y_1\cup Y_2\cup Y_3=Y^*_1\cup Y^*_2\cup Y^*_3,
\end{displaymath}
here we have
\textcolor{white}{xxxxxxxxxxxxxxxxxxxxxxxxxxxxxx}
\begin{equation}\label{eq9.10}
Y_1\cup\ldots\cup Y_s\ne Y^*_1\cup\ldots\cup Y^*_s,
\end{equation}
and both sides of \eqref{eq9.10} are proper subsets of~$\YYY$.
\end{remark}

Let us now choose a suitable $Y$-face $Y_1$ in~$\BBBB_1$.
We shall make use of the coordinate system given by \eqref{eq9.8} and \eqref{eq9.9}.

Figure~9.3 suggests that the $y$-coordinate of any point on $Y_1$ should be equal to $\NNNN-1$,
making it a $y$-distance $1$ below the top of the vertical wall separating $\BBBB_1$ and~$\BBBB_2$.
Figure~9.3 also suggests that $Y_1$ should contain a line segment $S_1$ in the $z$-direction
of length $1$ and which is mapped by the $\bfv$-flow to the top edge of this wall.
Clearly the $x$-coordinate of any point of $S_1$ must be equal to $\NNNN-\alpha\in(0,\NNNN)$.
This explains the requirement that $\NNNN>\alpha$ in \eqref{eq9.7}.

Now let $\bfp_1=(x_1,y_1,z_1)$ be the coordinates of the bottom left vertex of~$Y_1$.
The above argument shows that $x_1=[\NNNN-\alpha]$ and $y_1=\NNNN-1$.
We shall take
\begin{equation}\label{eq9.11}
z_1=0,
\end{equation}
to be explained later.

We next determine the $Y$-faces $Y_2,\ldots,Y_s$ in $\BBBB_2,\ldots,\BBBB_s$ and containing the line segments $S_2,\ldots,S_s$
in the $z$-direction respectively, so that the images of $Y_1,\ldots,Y_s$ under the modulo~$\NNNN$ analog of the natural projection \eqref{eq9.2} coincide,
as do the images of $S_1,\ldots,S_s$.

The argument thus far is summarized in the bottom half of Figure~9.7, where the $y$-direction is suppressed.

\begin{displaymath}
\begin{array}{c}
\includegraphics[scale=0.8]{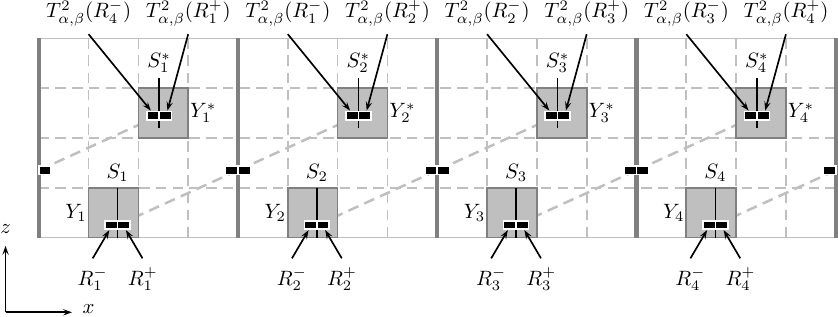}
\vspace{3pt}\\
\mbox{Figure 9.7: a process equivalent to short transportation}
\end{array}
\end{displaymath}

Consider next the image $S^*_2=S_1+(2\alpha,2,2\beta)$ of~$S_1$ that corresponds to two applications of the map $T_{\alpha,\beta}$
if we ignore the singularity at the top of the wall.

The $y$-coordinate of any point on $S^*_2$ is clearly $\NNNN+1$.
The $x$-coordinate is equal to $\NNNN+\alpha\in(\NNNN,2\NNNN)$.
The $z$-coordinate lies in the interval $(2\beta,1+2\beta)\subset(0,\NNNN)$.
Hence $S^*_2$ lies on two adjoining $Y$-faces in~$\BBBB_2$, with total length~$1$, and explains the choice \eqref{eq9.11}
as well as the condition $\NNNN>1+2\beta$ in \eqref{eq9.7}.

Now let $Y^*_2$ denote the $Y$-face that contains the longer half of the line segment~$S^*_2$,
with bottom left vertex $\bfp^*_2=(x^*_2,y^*_2,z^*_2)$.
Then $x^*_1=[\NNNN+\alpha]+1$, $y^*_2=\NNNN+1$ and
\begin{displaymath}
z^*_2=\left\{\begin{array}{ll}
[2\beta],&\mbox{if $\{2\beta\}<1/2$},\\
1+[2\beta],&\mbox{if $\{2\beta\}>1/2$}.
\end{array}\right.
\end{displaymath}

Let $S^*_2(0)=Y^*_2\cap S^*_2$ denote this longer part of $S^*_2$ that lies on the $Y$-face~$Y^*_2$,
and let $S_1(0)=S^*_2(0)-(2\alpha,2,2\beta)$ be the image of $S^*_2(0)$ under two applications
of the inverse map $T^{-1}_{\alpha,\beta}$ if we ignore the singularity at the top of the wall.
Then $S_1(0)$ is part of~$S_1$, and both $S_1(0)$ and $S^*_2(0)$ have length greater than~$1/2$.

We next determine the $Y$-faces $Y^*_1,Y^*_3,\ldots,Y^*_s$ in $\BBBB_1,\BBBB_3\ldots,\BBBB_s$
and containing the line segments $S^*_1(0),S^*_3(0),\ldots,S^*_s(0)$
in the $z$-direction respectively, so that the images of $Y^*_1,\ldots,Y^*_s$
under the modulo~$\NNNN$ analog of the natural projection \eqref{eq9.2} coincide,
as do the images of $S^*_1(0),\ldots,S^*_s(0)$.

This latter part of the argument is summarized in Figure~9.7, where the $y$-direction is suppressed.

\begin{lemma}\label{lem91}
There exist small special rectangles $R_1,\ldots,R_s,R^*_1,\ldots,R^*_s$ on the faces $Y_1,\ldots,Y_s,Y^*_1,\ldots,Y^*_s$ of~$\YYY$
such that the following conditions are satisfied:

\emph{(i)}
For every $\sigma=1,\ldots,s$, the set $R_\sigma$ satisfies
\begin{displaymath}
\frac{\lambda_2(R_\sigma\cap\WWW)}{\lambda_2(R_\sigma)}>\frac{99}{100}
\quad\mbox{or}\quad
\frac{\lambda_2(R_\sigma\cap\SSS)}{\lambda_2(R_\sigma)}>\frac{99}{100},
\end{displaymath}
and the set $R^*_\sigma$ satisfies
\begin{displaymath}
\frac{\lambda_2(R^*_\sigma\cap\WWW)}{\lambda_2(R^*_\sigma)}>\frac{99}{100}
\quad\mbox{or}\quad
\frac{\lambda_2(R^*_\sigma\cap\SSS)}{\lambda_2(R^*_\sigma)}>\frac{99}{100}.
\end{displaymath}

\emph{(ii)}
The images on $\MMM_0$ of $R_1,\ldots,R_s$ under the modulo~$\NNNN$ analog of the projection \eqref{eq9.2} coincide.

\emph{(iii)}
For every $\sigma=1,\ldots,s$, the set $R^*_\sigma$ satisfies
\begin{displaymath}
R^*_\sigma=T^2_{\alpha,\beta}(R_{\sigma-1}^-)\cup T^2_{\alpha,\beta}(R_\sigma^+),
\end{displaymath}
with the convention that $R_0^-=R_s^-$.
\end{lemma}

As in our earlier discussion concerning Theorem~\ref{thm1}, the dominant colours in the rectangles
$R_1,\ldots,R_s$ is the same colour $\CCC$ which is precisely one of $\WWW$ and~$\SSS$.
In view of (ii), this clearly contradicts \eqref{eq9.5} and \eqref{eq9.6},
and establishes the ergodicity of $T_{\alpha,\beta}$.

The proof of Lemma~\ref{lem91} goes along similar lines to the proof of Lemma~\ref{lem12},
apart from some modification of Lemma~\ref{lem32}.
There, in order to find non-defective $T_{\alpha,\beta}$-power chains \eqref{eq3.13}
and non-defective $T^{-1}_{\alpha,\beta}$-power chains \eqref{eq3.22},
we look for an integer $j$ in the range $q_{k+1}+1\le j\le q'_{h+1}-q_{k+1}$.
In that particular situation where the $z$-direction in integrable, every integer $j$
in this range is a candidate.

Here we have to be more careful, and consider only those values of $j$ such that
the first term in the $T_{\alpha,\beta}$-power chain \eqref{eq3.13} has non-empty intersection with the line segment $S^*_\sigma(0)$,
and the first term in the $T^{-1}_{\alpha,\beta}$-power chain \eqref{eq3.22} has non-empty intersection with the line segment $S_\sigma(0)$.
Let us refer to these values of $j$ as \textit{special}.
Since the line segments $S_\sigma(0)$ and $S^*_\sigma(0)$ have length greater than~$1/2$,
it follows from the $2$-distance theorem that the proportion of special values of $j$ in the prescribed range
is greater than a positive constant dependent only on the parameter~$h$.
The argument in the proof of Lemma~\ref{lem12} can be adapted easily to accommodate this loss of a factor.

%
%

\end{document}